\documentclass[12pt,twoside]{amsart}
\usepackage{amsmath, amsthm, amssymb,bm}
\usepackage{paralist}  %For compact itemize and enumerate environments
\usepackage{url}
\usepackage{array}
\usepackage{calc}
\usepackage{colonequals}

\evensidemargin \oddsidemargin
\newtheorem{thm}{Theorem}[section]
\newtheorem{theorem}[thm]{Theorem}
\newtheorem*{nnthm}{Theorem}

\newtheorem{lemma}[thm]{Lemma}

\newtheorem{prop}[thm]{Proposition}

\theoremstyle{definition}

\theoremstyle{remark}
\newtheorem{remark}[thm]{Remark}

\newtheorem{example}[thm]{Example}

\title[Local Selectivity of Orders in Central Simple Algebras]
{Local Selectivity of Orders in Central Simple Algebras}

\author {Benjamin Linowitz}

\address{Department of Mathematics\\ 
530 Church Street\\
University of Michigan\\
Ann Arbor, MI 48109}

\email[] {linowitz@umich.edu}
\urladdr{http://www-personal.umich.edu/~linowitz/home.html}

\author {Thomas R. Shemanske}

\address{Department of Mathematics\\ 
6188 Kemeny Hall\\
Dartmouth College\\
Hanover, NH 03755}

\email[] {thomas.r.shemanske@dartmouth.edu}
\urladdr{http://www.math.dartmouth.edu/~trs/ }

\date{\today}

\newcommand{\ds}{\displaystyle}

\newcommand{\ba}{{\mathbf a}}
\newcommand{\bb}{{\mathbf b}}
\newcommand{\bc}{{\mathbf c}}

\newcommand{\D}{{\mathcal D}}
\newcommand{\E}{{\mathcal E}}

\renewcommand{\L}{{\mathcal L}}
\newcommand{\M}{{\mathcal M}}
\newcommand{\fM}{{\mathfrak M}}

\newcommand{\N}{{\mathcal N}}

\newcommand{\fN}{{\mathfrak N}}

\renewcommand{\O}{{\mathcal O}} % The letter O
\newcommand{\0}{{\mathbf 0}}    % The digit 0
\newcommand{\fP}{{\mathfrak P}}
\newcommand{\fp}{{\mathfrak p}}

\newcommand{\bpi}{{\bm\pi}}

\newcommand{\Q}{\mathbb{Q}}
\newcommand{\R}{{\mathcal R}}
\newcommand{\Z}{\mathbb{Z}}

\newcommand{\la}{\langle}
\newcommand{\ra}{\rangle} 
\newcommand{\wH}{{\widehat H}}
\newcommand{\wL}{{\widehat L}}

\newcommand{\End}{\operatorname{End}}
\newcommand{\Gal}{\operatorname{Gal}}

\newcommand{\lcm}{\operatorname{lcm}}
\newcommand{\ord}{\operatorname{ord}}
\newcommand{\diag}{\operatorname{diag}}
\newcommand{\GL}{\operatorname{GL}}
\newcommand{\SL}{\operatorname{SL}}
\newcommand{\Mat}{\operatorname{M}}

\parskip=\medskipamount

\begin{document}



\begin{abstract}
  Let $B$ be a central simple algebra of degree $n$ over a number
  field $K$, and $L\subset B$ a strictly maximal subfield.  We say
  that the ring of integers $\O_L$ is \textit{selective} if there
  exists an isomorphism class of maximal orders in $B$ no element of
  which contains $\O_L$.  Many authors have worked to characterize the
  degree to which selectivity occurs, first in quaternion algebras,
  and more recently in higher-rank algebras.  In the present work, we
  consider a local variant of the selectivity problem and
  applications.

  We first prove a theorem characterizing which maximal orders in a
  local central simple algebra contain the global ring of integers
  $\O_L$ by leveraging the theory of affine buildings for $SL_r(D)$
  where $D$ is a local central division algebra.  Then as an
  application, we use the local result and a local-global principle to
  show how to compute a set of representatives of the isomorphism
  classes of maximal orders in $B$, and distinguish those which are
  guaranteed to contain $\O_L$.  Having such a set of representatives
  allows both algebraic and geometric applications.  As an algebraic
  application, we recover a global selectivity result mentioned above,
  and give examples which clarify the interesting role of partial
  ramification in the algebra.

\end{abstract}

\maketitle

\section{Introduction}

Let $B$ be a central simple algebra of degree $n$ over a number field
$K$, and $L \subset B$ a strictly maximal (i.e., $[L:K] = n$) subfield of
$B$.  There exists at least one maximal order $\R$ of $B$ which
contains the ring of integers $\O_L$, and so every element of the
isomorphism class of $\R$ admits an embedding of $\O_L$. If there
exists an isomorphism class of maximal orders in $B$ no element of
which contains $\O_L$, then $\O_L$ is said to be \textit{selective}.
This is equivalent to no element of the isomorphism class admitting an
embedding of $\O_L$.

Many authors worked to characterize the degree to which selectivity
occurs: \cite{Chinburg-Friedman}, \cite{Chan-Xu},\cite{Guo-Qin},
\cite{Maclachlan}, \cite{Linowitz-selectivity} (in quaternion
algebras), and \cite{Arenas-Carmona-2003},
\cite{Linowitz-Shemanske-EmbeddingPrimeDegree},
\cite{Arenas-Carmona-2011-Representation-Fields},
\cite{Arenas-Carmona-2014-Selectivity-Division-Algebras} (in
higher-rank algebras).  The tools which have been employed vary from
the Bruhat-Tits tree in \cite{Chinburg-Friedman}, to representation
fields (a subfield of a spinor class field) in
\cite{Arenas-Carmona-2011-Representation-Fields}.  The results of
\cite{Arenas-Carmona-2011-Representation-Fields} are very general,
offering the proportion of isomorphism classes of maximal orders (an
element of) which contain the order $\O_L$ (or any of its suborders),
in terms of the index of the representation field in an associated
spinor class field.

In the present work, the authors continue their study (cf.
\cite{Linowitz-Shemanske-EmbeddingPrimeDegree}) of how the use of
Bruhat-Tits buildings can illuminate problems for higher rank algebras
as the Bruhat-Tits tree was used to answer selectivity questions in
the quaternion case \cite{Chinburg-Friedman}.  Locally, since all
maximal orders are conjugate, every maximal order admits an embedding
of $\O_L$, so a local selectivity question must be more discerning: to
characterize the maximal orders in a local central simple algebra
which contain $\O_L$.  This turns out to be both an interesting and
somewhat nuanced question.

To be more precise, we set some notation.  By Wedderburn's structure
theorem, we shall assume that $B = M_r(D)$ where $D$ is a central
division algebra over $K$ of degree $m$, so that $n = rm$. Let $\nu$
be any place of $K$, and denote the completion of $K$ at $\nu$ by
$K_\nu$, and when $\nu$ is a finite place, denote by $\O_\nu$ the
valuation ring of $K_\nu$. The completion of $B$ at $\nu$ is given by
\[B_\nu = K_\nu\otimes_K B \cong M_{r_\nu}(D_\nu),\]
where $D_\nu$ is a central division algebra over $K_\nu$ of degree
$m_\nu$, so that $n = rm = r_\nu m_\nu$ with $r \mid r_\nu$.  We say
that a place $\nu$ of $K$ \textit{splits} in $B$ if $m_\nu = 1$,
\textit{totally ramifies} in $B$ if $m_\nu = n$, and \textit{partially
  ramifies} in $B$ if $1 < m_\nu < n$. If $B$ totally ramifies at a
finite place $\nu$, there is a unique maximal order of $B_\nu$, so of
course $\O_L$ is contained in it, thus the only interest in local
selectivity arises when $\nu$ is not totally ramified.

As a consequence of the condition that $L$ is a strictly maximal
subfield of $B$, we know that for each place $\nu$ of $K$ and for all
places $\fP$ of $L$ lying above $\nu$, $m_\nu \mid [L_\fP:K_\nu]$ (the
Albert-Brauer-Hasse-Noether theorem), we have by (31.10) of
\cite{Reiner-book}, that each $L_\fP$ splits $D_\nu$ (and hence
$B_\nu$), and moreover by (28.5) of \cite{Reiner-book}, for each place
$\fP$ of $L$ with $\fP \mid \nu$, there is a smallest integer
$r_\fP\ge 1$ so that $L_\fP$ embeds in $M_{r_\fP}(D_\nu)$ as a
$K_\nu$-algebra; here $r_\fP = [L_\fP:K_\nu]/m_\nu$.
Theorem~\ref{thm:OLembedslocally} (which applies even in the
quaternion case) says:

\begin{nnthm}
  Let $B$ be a central simple algebra over a number field $K$ of
  dimension $n^2 \ge 4$ and $L$ a degree $n$ field extension of $K$
  which is contained in $B$. Let $\nu$ be a finite place of $K$ which
  splits or is partially ramified in $B$, so
  $B_\nu = M_{r_\nu}(D_\nu)$ with $r_\nu > 1$, and $D_\nu$ a central
  division algebra over $K_\nu$ of degree $m_\nu$. Assume that the
  place $\nu$ is unramified in $L$, and let $\{\fP_1, \dots, \fP_g\}$
  be the set of places of $L$ lying above $\nu$.  As above, let
  $r_{\fP_i} = [L_{\fP_i}: K_\nu]/m_\nu$.  Then $\O_L$ is contained in
  the maximal orders of $B_\nu$ represented by the homothety class
  $[\mathcal L] =[a_1, \dots, a_{r_\nu}] \in \Z^{r_\nu}/\Z(1,\dots,1)$
  if and only if there are $\ell_i \in \Z$ such that
  $[\mathcal L] = [\underbrace{\ell_1,
    \dots,\ell_1}_{r_{\fP_1}},\underbrace{\ell_2,
    \dots,\ell_2}_{r_{\fP_2}},\dots, \underbrace{\ell_g,
    \dots,\ell_g}_{r_{\fP_g}}]$.
\end{nnthm}

There are a number of applications of such a local result.  As a
primary application, it allows us to construct a set of
representatives of all the isomorphism classes of maximal orders in
the global algebra $B$, and specify those which are guaranteed to
contain $\O_L$.

This is turn has at least two other applications, one algebraic and
one geometric.  In terms of the global selectivity problem, it allows
one to compute not simply the selectivity proportion for $\O_L$, but
distinguish those classes which necessarily admit an embedding of
$\O_L$.  The same computation of explicit representatives of maximal
orders in $B$ also can be used in geometric realms such the
development of a higher-dimensional analog of a construction of
Vigner\'as \cite{Vigneras-isospectral} of isospectral non-isometric
Riemann surfaces (e.g., \cite{Lubotzky-Samuels-Vishne}).  Explicit
characterization of these maximal orders allows the geometry of the
corresponding manifold to be detailed, e.g., computation of the
geodesic length spectrum.

As an algebraic application, we recover a global selectivity result
mentioned above, and give an explicit example which demonstrates the
effect of partial ramification in the algebra.

\section{Local Results}\label{sec:localtheory} Because our main
application of these local results will be to construct a
distinguished set of representatives for the isomorphism classes of
maximal orders in the global algebra $B$, and then recover a
selectivity result, we retain global notation throughout to allow for
some dovetailing of local and global remarks.

Let $\nu$ be a finite place of $K$, and
$B_\nu \cong M_{r_\nu}(D_{\nu})$, with $D_\nu$ a central division
algebra of degree $m_\nu$ over $K_\nu$.  Recall
$[L:K] = n = \deg_K(B) = rm = r_\nu m_\nu$.  We shall note in
Theorem~\ref{thm:divisionalgebra-noselectivity}, if there is a finite
place $\nu$ for which $B_\nu$ is a division algebra ($r_\nu = 1$),
there can be no selectivity, so we assume for this section that
$r_\nu > 1$.  

Recall from the introduction that for each place $\nu$ of $K$, and
each place $\fP$ of $L$ lying above $\nu$,
$r_\fP = [L_\fP:K_\nu]/m_\nu \ge 1$ is the smallest integer so that
$L_\fP$ embeds in $M_{r_\fP}(D_\nu)$ as a $K_\nu$-algebra.

We note that
\[\sum_{\fP\mid \nu} r_\fP = \sum_{\fP\mid \nu}
\frac{[L_\fP:K_\nu]}{m_\nu} = \frac{[L:K]}{m_\nu} = \frac{n}{m_\nu} =
r_\nu,
\]
and this means that

\begin{equation}
  \label{eq:1}
  K_\nu \otimes_K L \cong \bigoplus_{\fP\mid \nu} L_\fP \hookrightarrow
  \bigoplus_{\fP\mid \nu}M_{r_\fP}(D_\nu) \hookrightarrow M_{r_\nu}(D_\nu),
\end{equation}
with the last embedding as blocks along the diagonal.

We have fixed a global maximal order $\R$ in $B$ which contains
$\O_L$.  We define completions $\R_\nu \subseteq B_\nu$ by:
\[
\R_\nu =
\begin{cases}
  \O_\nu \otimes_{\O_K} \R&\textrm{if $\nu$ is finite}\\
  K_\nu\otimes_{\O_K} \R  = B_\nu&\textrm{if $\nu$ is infinite.}\\
\end{cases}
\]
For finite places $\nu$, we know by (17.3) of \cite{Reiner-book},
that $\R_\nu$ is conjugate to $M_{r_\nu}(\Delta_\nu)$ where
$\Delta_\nu$ is the unique maximal order of $D_\nu$, so we assume that
$B_\nu$ has been identified with $M_{r_\nu}(D_\nu)$ in such a way that
$\R_\nu = M_{r_\nu}(\Delta_\nu)$.  Since all maximal orders of
$M_{r_\fP}(D_\nu)$ are conjugate to $M_{r_\fP}(\Delta_\nu)$ we may, by
a change of basis, adjust the embeddings
$L_{\fP} \hookrightarrow M_{r_\fP}(D_\nu)$ so that the ring of
integers $\O_\fP \subset M_{r_\fP}(\Delta_\nu)$.  Now by Exercise 5.4
(p. 76) of \cite{Reiner-book}, $\O_\nu$ is a faithfully flat
$\O_K$-module and the containment of $\O_L \subset \R$ extends to one
of
$\O_\nu\otimes_{\O_K} \O_L \subset \O_\nu\otimes_{\O_K} \R = \R_\nu$.
We identify $\O_L$ with its image $1\otimes \O_L$, so will simply
write $\O_L \subset \R_\nu$.  More precisely, we will identify $\O_L$
with its image in $\bigoplus_{\fP\mid \nu} \O_\fP$ via:
\begin{equation}
  \label{eq:integralblockstructure}
  \O_L \subset \O_\nu\otimes_{\O_K} \O_L \hookrightarrow
  \bigoplus_{\fP\mid \nu} \O_\fP \subset \bigoplus_{\fP\mid \nu}
  M_{r_\fP}(\Delta_\nu) \subset M_{r_\nu}(\Delta_\nu) = \R_\nu,
\end{equation}
where we are using the subset notation to identify the object and its
image.

Fix a uniformizing parameter $\bpi = \bpi_{D_\nu}$ of the maximal
order $\Delta_\nu$, and let \newline
$d_k^\ell =
\diag(\underbrace{\bpi^\ell,\dots,\bpi^\ell}_k,1,\dots,1)\in
M_{r_\nu}(D_\nu)$.  Put
\begin{equation}\label{eq:blockstr}
  \R(k,\ell) \colonequals d_{k}^\ell \R_\nu d_k^{-\ell} = d_{k}^\ell M_{r_\nu}(\Delta_\nu) d_k^{-\ell} = 
  \begin{pmatrix}
    M_k(\Delta_\nu)& \bpi^{\ell}M_{k\times r_\nu-k}(\Delta_\nu)\\
    \bpi^{-\ell}M_{r_\nu-k\times k}(\Delta_\nu)& M_{r_\nu-k}(\Delta_\nu)\\
  \end{pmatrix}\subset M_{r_\nu}(D_\nu).
\end{equation}

Note that
$\R(0,\ell) = \R(r_\nu,\ell) =\R(k,0) =\R_\nu= M_{r_\nu}(\Delta_\nu)$.
If we let $\fP_1, \dots, \fP_g$ denote all the places of $L$ lying
above $\nu$, then from
equations~(\ref{eq:integralblockstructure}),(\ref{eq:blockstr}) above,
it is evident that for all $\ell_1, \dots, \ell_g \in \Z$,
\begin{equation}
  \O_L \subset \R(r_{\fP_1},\ell_1) \ \cap \R(r_{\fP_1}+r_{\fP_2},\ell_2) 
  \cap \cdots \cap \R(r_{\fP_1}+\cdots+r_{\fP_g},\ell_g), 
\end{equation}
that is,
\begin{equation}\label{eq:orderscontainingomega}
  \O_L \subset \bigcap_{\ell_i \in \Z}\big[
  \R(r_{\fP_1},\ell_1) \ \cap \R(r_{\fP_1}+r_{\fP_2},\ell_2) 
  \cap \cdots \cap \R(r_{\fP_1}+\cdots+r_{\fP_g},\ell_g) \big] = 
  \bigoplus_{\fP\mid \nu} M_{r_\fP}(\Delta_\nu)\subset \R_\nu.
\end{equation}

\subsection{Affine buildings and type distance.}
We now translate this to the language of affine buildings.  By (17.4)
of \cite{Reiner-book}, we know that every maximal order in $B_\nu$ has
the form $\End_{\Delta_\nu}(\Lambda)$ where $\Lambda$ is a full (i.e.,
rank $r_\nu$), free (left) $\Delta_\nu$-lattice in $D_\nu^{r_\nu}$.
We recall that a maximal order is characterized completely by the
homothety class of its associated lattice, and homothety classes of
lattices in $D_\nu^{r_\nu}$ are a very concrete way in which to
characterize the vertices of the affine building associated to
$\SL_{r_\nu}(D_\nu)$ (see section 3 of \cite{Abramenko-Nebe}, or
\cite{Ronan} Ch.9, \S2). We know $\GL_{r_\nu}(D_\nu)$ acts
transitively on the free $\Delta_\nu$-lattices of rank $r_\nu$ and
acts invariantly on the homothety classes.  Using that the maximal
order $\Delta_\nu$ of $D_\nu$ is a discretely valued ring with
$\bpi = \bpi_{D_\nu}$ a uniformizer, we put $\ord_{\bpi}$ to be the
exponential valuation on $D_\nu$. Then we note that $\ord_{\bpi}$ is
trivial on the commutator $[D^\times,D^\times]$, so for each
$g \in \GL_{r_\nu}(D_\nu)$, $\ord_{\bpi}(\det(g))$ is a well-defined
integer, where $\det(\cdot)$ is the Dieudonn\'e determinant.  It is
then natural to define the \textit{type} of a vertex as an integer
modulo $r_\nu$ as follows (see \cite{Ronan}).  Let $\Lambda$ be a
(free of rank $r_\nu$) $\Delta_\nu$-lattice whose homothety class is
assigned the type 0. For another such lattice $\Gamma$, let $g$ be any
element of $\GL_{r_\nu}(D_\nu)$ so that $\Gamma = g(\Lambda)$.  Then
the class of $\Gamma$ is assigned type
$\ord_{\bpi}(\det(g)) \pmod {r_\nu}$, which is well-defined on the
homothety class since we are viewing the type modulo~$r_\nu$.

The simplicial structure of the building associated to
$\SL_{r_\nu}(D_\nu)$ is reflected through its vertex types. In
particular, the $r_\nu$ vertices of any chamber have types $0$ through
$(r_\nu-1)$.  In relating the vertices, we utilize the invariant
factor theory which applies to free $\Delta_\nu$-lattices of
rank~$r_\nu$.  Let $\Gamma$ and $\Lambda$ be two rank $r_\nu$ free
$\Delta_\nu$-lattices. Since we are working with homothety classes, we
may assume that $\Gamma \subseteq \Lambda$.  By (17.7) of
\cite{Reiner-book}, given two such lattices, there exists a basis
$\{e_1, \dots, e_{r_\nu}\}$ of $D_\nu^{r_\nu}$ and rational integers
$0\le a_1 \le \cdots \le a_r$ so that
\[\Lambda = \bigoplus_{i=1}^{r_\nu} \Delta_\nu e_i
\quad\mbox{and}\quad \Gamma = \bigoplus_{i=1}^{r_\nu} \Delta_\nu
\bpi^{a_i} e_i.
\]
Suppose that $\E = \End_{\Delta_\nu}(\Lambda)$, and
$\E' = \End_{\Delta_\nu}(\Gamma)$. Using the invariant factor
decomposition above, we define the \textit{type distance}
$td_\nu(\E, \E')$ to be
\[td_\nu(\E, \E') = \sum_{i=1}^{r_\nu} a_i \pmod {r_\nu}.
\]
We note that this definition depends only upon the homothety class of
the lattices. While it is true that
$td_\nu(\E, \E') \equiv -td_\nu(\E', \E)\pmod {r_\nu}$, our main
concern will be when the type distance $td_\nu(\E, \E')$ is divisible
by some integer, so the order will be of little consequence.  This
definition of type distance generalizes the one in
\cite{Linowitz-Shemanske-EmbeddingPrimeDegree}, where whenever the
algebra was not totally ramified, it was split, so that $r_\nu = n$
and $D_\nu = K_\nu$.

\subsection{Local selectivity}

Pick a basis $\{\alpha_1, \dots, \alpha_{r_\nu}\}$ of $D_\nu^{r_\nu}$
(and hence in particular fix an apartment of the building associated
to $SL_{r_\nu}(D_\nu)$), so that with respect to this basis,
$\R_\nu = M_{r_\nu}(\Delta_\nu) = \End_{\Delta_\nu}(\Lambda)$, where
$\Lambda = \oplus_{i=1}^{r_\nu} \Delta_\nu \alpha_i$. Also we have
$\R(k,\ell) = \End_{\Delta_\nu}(\M(k,\ell))$ where
$\M(k,\ell) = \oplus_{i=1}^{k}\bpi^\ell\Delta_\nu \alpha_i \oplus
\oplus_{i={k+1}}^{r_\nu}\Delta_\nu\alpha_i$
and $\bpi$ is our fixed uniformizer in $\Delta_\nu$.  As usual, this
maximal order in $B_\nu$ can be represented by the homothety class of
the lattice $\M(k,\ell)$,
$[\M(k,\ell)] \colonequals [\underbrace{\ell,\dots,\ell}_k,0,\dots,0]
\in \Z^{r_\nu}/\Z(1,\dots,1)$.
Observe that $\R(k,\ell)$ has type $k\ell\pmod {r_\nu}$.

With the notation fixed as above, we characterize precisely which
maximal orders in this apartment contain $\O_L$. The theorem is valid
even in the quaternion case ($n=2$).  We shall continue to assume
$r_\nu > 1$ (i.e. $\nu$ not totally ramified in $B$), otherwise
$B_\nu$ has a unique maximal order, which must clearly contain $\O_L$.

\begin{thm}\label{thm:OLembedslocally}
  Let $B$ be a central simple algebra over a number field $K$ of
  dimension $n^2 \ge 4$ and $L$ a degree $n$ field extension of $K$
  which is contained in $B$. Let $\nu$ be a finite place of $K$ which
  splits or is partially ramified in $B$, so
  $B_\nu = M_{r_\nu}(D_\nu)$ with $r_\nu > 1$, and $D_\nu$ a central
  division algebra over $K_\nu$ of degree $m_\nu$. Assume that the
  place $\nu$ is unramified in $L$, and let $\{\fP_1, \dots, \fP_g\}$
  be the set of places of $L$ lying above $\nu$.  As above, let
  $r_{\fP_i} = [L_{\fP_i}: K_\nu]/m_\nu$.  Then $\O_L$ is contained in
  the maximal orders of $B_\nu$ represented by the homothety class
  $[\mathcal L] =[a_1, \dots, a_{r_\nu}] \in \Z^{r_\nu}/\Z(1,\dots,1)$
  if and only if there are $\ell_i \in \Z$ such that
  $[\mathcal L] = [\underbrace{\ell_1,
    \dots,\ell_1}_{r_{\fP_1}},\underbrace{\ell_2,
    \dots,\ell_2}_{r_{\fP_2}},\dots, \underbrace{\ell_g,
    \dots,\ell_g}_{r_{\fP_g}}]$.
  
\end{thm}

\begin{proof}[Proof of Theorem]
  Consider equation~(\ref{eq:orderscontainingomega}).  We know that
  $\O_L$ is contained in
  $\R(r_{\fP_1},\ell_1) \ \cap \R(r_{\fP_1}+r_{\fP_2},\ell_2) \cap
  \cdots \cap \R(r_{\fP_1}+\cdots+r_{\fP_g},\ell_g)$
  for any choice of $\ell_i\in \Z$.  These orders correspond to
  homothety classes of lattices
  $[\M(r_{\fP_1}+\cdots+r_{\fP_i},\ell_i)]=\ell_i
  [\M(r_{\fP_1}+\cdots+r_{\fP_i},1)]=\ell_i
  [\underbrace{1,\dots,1}_{r_{\fP_1}+\cdots+r_{\fP_i}},0,\dots,0]$
  as an element of $\Z^{r_\nu}/\Z(1,\dots,1)$.  In
  \cite{Ballantine-Rhodes-Shemanske}, it is shown that walks in an
  apartment are consistent with the natural group action on
  $\Z^{r_\nu}/\Z(1,\dots,1)$, and by \cite{Shemanske-Split} the
  intersection of any finite number of maximal orders (containing
  $\Delta_\nu^{r_\nu}$) in an apartment is the same as the
  intersection of all the maximal orders in the convex hull they
  determine.  The references above discuss the case where
  $D_\nu = K_\nu$, but the arguments generalize trivially to the
  setting of a vector space over $D_\nu$ instead of $K_\nu$, as does
  the theory of buildings.  Using these observations, we deduce that
  $\O_L$ is contained in maximal orders corresponding to
  \begin{align*}
    [\M(r_{\fP_1},\ell_1) &+ \M(r_{\fP_1}+r_{\fP_2}, \ell_2)
                            + \cdots + \M(r_{\fP_1} + \cdots + r_{\fP_g},\ell_g)] = \\
                          & [\underbrace{\ell_1+\cdots+\ell_g,
                            \dots,\ell_1+\cdots+\ell_g}_{r_{\fP_1}},\underbrace{\ell_2+\cdots+\ell_g,
                            \dots,\ell_2+\cdots+\ell_g}_{r_{\fP_2}},\dots, \underbrace{\ell_g,
                            \dots,\ell_g}_{r_{\fP_g}}].
  \end{align*}

  Since the $\ell_i\in \Z$ are arbitrary, a simple change of variable
  ($\ell_k + \cdots+ \ell_g \mapsto \ell_k$) shows that $\O_L$ is
  contained in the maximal orders specified in the proposition.  We
  now show these are the only maximal orders in the apartment which
  contain $\O_L$. To proceed,  we need to set some notation and
  prove a technical lemma.

  For any place $\fP$ of $L$ lying above $\nu$, we have (by
  assumption) that $L_\fP/K_\nu$ is an unramified extension of degree
  $f \colonequals r_\fP m_\nu$.  Let $\overline\O_\fP$ and
  $\overline \O_\nu$ be the associated residue fields and let
  $q = |\overline \O_\nu|$.  Now let $\omega$ be a primitive $q^f-1$
  root of unity over $K_\nu$, so that $L_\fP = K_\nu(\omega)$.  We
  know that $D_\nu$ contains an inertia field, $W_\nu$, which is
  unique up to conjugacy.  It is an unramified extension of $K_\nu$
  and a maximal subfield of $D_\nu$, having degree
  $[W_\nu:K_\nu] = m_\nu$. Without loss, we may assume that
  $K_\nu \subseteq W_\nu \subseteq L_\fP$, with $W_\nu$ generated over
  $K_\nu$ by an appropriate power of $\omega$ (since
  $q^{m_\nu-1} \mid q^f-1)$.  Now let $h = \min_{K_\nu}(\omega)$ be
  the minimal polynomial of $\omega$ over $K_\nu$.  As $\omega$ is
  integral, we know $h \in \O_\nu[x]$.

  \begin{prop}\label{prop:unramifiedlocalextensions}
    From Theorem 5.10 and Corollary 5.11 of \cite{Reiner-book}, we
    recall
    \begin{compactitem}
    \item
      $\O_\fP = \O_\nu[\omega]; \quad\overline \O_\fP =
      \overline\O_\nu[\overline \omega]$.
    \item $\overline h = \min_{\overline \O_\nu}(\overline \omega)$
      and is separable.
    \item $L_\fP/K_\nu$ and $\overline\O_\fP/\overline \O_\nu$ are
      cyclic extensions with isomorphic Galois groups.
    \end{compactitem}
  \end{prop}

  The technical lemma we require is:
  \begin{lemma}\label{lemma:IntegralClosure}
    Let $R$ be the valuation ring of $\nu$ in $K$, and $S$ its
    integral closure in $L$.  Suppose that $\E$ is a ring containing
    both $R$ and $\O_L$.  Then $S \subset \E$.
  \end{lemma}
  \begin{proof}
    By Corollary 5.22 of \cite{Atiyah-Macdonald}, $S$ is the
    intersection of all valuation rings of $L$ which contain $R$. The
    valuation ring $R$ is equal to the localization $D^{-1}\O_K$ where
    $D = \O_K \setminus \nu\O_K$, and the valuation rings of $L$ which
    contain $R$ are precisely the localizations of $\O_L$ at the
    places $\fP_1, \dots, \fP_g$ of $L$ which lie above $\nu$. By p43
    of \cite{Reiner-book}, the intersection of these localizations,
    $S$, is equal to the localization $T^{-1}\O_L$, where
    $T = \O_L \setminus (\fP_1\cup \cdots \cup \fP_g)$.

    It is easy to see that $D \subseteq T$ since if
    $\alpha \in D= \O_K\setminus \nu\O_K$, we have $\alpha \in \O_L$
    and if $\alpha \in \fP_i$ for some $i$, then
    $\alpha \in \fP_i\cap \O_K = \nu\O_K$, a contradiction.  So
    $D^{-1}\O_L \subseteq T^{-1}\O_L$.  To show equality, we need only
    show that for $\beta \in T$, $\beta^{-1} \in D^{-1}\O_L$. Since
    $\beta \in \O_L$ we know that
    $N_{L/K}(\beta)\colonequals \beta \tilde\beta \in \O_K$, which
    means $\tilde \beta = \beta^{-1} N_{L/K}(\beta) \in L$ and is
    integral, hence in $\O_L$, so
    $\beta^{-1} = \tilde\beta/N_{L/K}(\beta)$, and we need only check
    that $N_{L/K}(\beta) \in D$.  Suppose to the contrary that
    $N_{L/K}(\beta) \in \nu\O_K$.  Then we show that $\nu \in \fP_i$
    for some $i$, a contradiction.  We use the extension of the norm
    to ideals and that $N_{L/K}(\beta \O_L) = N_{L/K}(\beta)\O_K$. If
    we write $\beta\O_L = \fP_1^{m_1}\cdots \fP_g^{m_g} \mathfrak Q$
    where $\mathfrak Q$ is an ideal place to the $\fP_i$, and let
    $\mathfrak q = \mathfrak Q\cap \O_K$, then
    $N_{L/K}(\beta\O_L) = (\nu\O_K)^{\sum_{i=1}^g m_if_i}\mathfrak
    q^f$
    where $f_i = f(\fP_i:\nu)$ and $f= f(\mathfrak Q:\mathfrak q)$ are
    the corresponding inertial degrees.  So
    $N_{L/K}(\beta) \in \nu\O_K$ if and only if some $m_i >0$ which is
    to say that $\beta\in \fP_i$, a contradiction.  Thus we have that
    $S$, the integral closure of $R$ in $L$, can be expressed as
    $D^{-1}\O_L= R\cdot \O_L$, so any ring $\E$ containing $R$ and
    $\O_L$ contains $S$.
  \end{proof}

  Continuing now with the proof of Theorem~\ref{thm:OLembedslocally},
  denote $\fp$ denote the two-sided ideal $\bpi\Delta_\nu$ of
  $\Delta_\nu$, and suppose that $\O_L$ is contained in a maximal
  order $\Lambda(a_1, \dots, a_{r_\nu})$ where
  \begin{align*}
    \Lambda(a_1, \dots, a_{r_\nu}) &= \diag(\bpi^{a_1}, \dots,
                                     \bpi^{a_{r_\nu}})M_{r_\nu}(\Delta_\nu)\diag(\bpi^{a_1}, \dots, \bpi^{a_{r_\nu}})^{-1} =\\
                                   &\begin{pmatrix}
                                     \Delta_\nu&\fp^{a_1 - a_2}&\fp^{a_1-a_3}& \dots&\fp^{a_1-a_{r_\nu}}\\
                                     \fp^{a_2-a_1}&\Delta_\nu&\fp^{a_2-a_3}&\dots&\fp^{a_2 - a_{r_\nu}}\\
                                     \fp^{a_3-a_1}&\fp^{a_3-a_2}&\ddots&\dots&\fp^{a_3 - a_{r_\nu}}\\
                                     \vdots&\vdots&&\Delta_\nu&\vdots\\
                                     \fp^{a_{r_\nu} - a_1}&\dots&&\fp^{a_{r_\nu} - a_{r_\nu-1}}&\Delta_\nu
                                   \end{pmatrix},
  \end{align*}
  that is $\Lambda(a_1, \dots, a_{r_\nu})$ corresponds to the
  homothety class of the lattice $[a_1, \dots, a_{r_\nu}]$ relative to
  our fixed basis $\{\alpha_1, \dots, \alpha_{r_\nu}\}$ of
  $D_\nu^{r_\nu}$.  By equation~(\ref{eq:orderscontainingomega}), we
  can reorder subsets of the basis
  $\{\alpha_1,\dots, \alpha_{r_{\fP_1}}\}$,
  $\{\alpha_{r_{\fP_1}+1}, \dots, \alpha_{r_{\fP_1}+r_{\fP_2}}\}$,
  \dots,
  $\{\alpha_{r_{\fP_1}+\dots+r_{\fP_{g-1}}+1}, \dots,
  \alpha_{r_\nu}\}$
  so that equation~(\ref{eq:orderscontainingomega}) remains valid and
  $a_1 \le \cdots\le a_{r_{\fP_1}}$,
  $a_{r_{\fP_1}+1} \le \cdots \le a_{r_{\fP_1}+r_{\fP_2}}$, \dots,
  $a_{r_{\fP_1}+\cdots+r_{\fP_{g-1}}+1} \le \cdots \le a_{r_\nu}$.

  Now we assume that $[a_1, \dots, a_{r_\nu}]$ is not of the form
  $[\underbrace{\ell_1, \dots,\ell_1}_{r_{\fP_1}},\underbrace{\ell_2,
    \dots,\ell_2}_{r_{\fP_2}},\dots, \underbrace{\ell_g,
    \dots,\ell_g}_{r_{\fP_g}}]$
  for $\ell_i \in \Z$.  Since we can permute the order in which we
  list the places $\fP_i$ of $L$ lying above $\nu$, we may assume that
  there is an $r_0$ with $1 \le r_0 < r_{\fP_1}$ so that
  $a_1 = \cdots = a_{r_0}< a_{r_0+1} \le \cdots \le a_{r_{\fP_1}}$.
  From equation~(\ref{eq:integralblockstructure}), we know that
  $\O_L \subset \oplus_{i=1}^g\O_{\fP_i} \subset \oplus_{i=1}^g
  M_{r_{\fP_i}}(\Delta_\nu) \subset M_{r_\nu}(\Delta_\nu)$, so
  \[\O_L \subset \oplus_{i=1}^g
  M_{r_{\fP_i}}(\Delta_\nu)\cap \Lambda(a_1, \dots,a_{r_\nu}) =:
  \Gamma = \renewcommand\arraystretch{1.5} \left(\begin{array}{c|c|c}
      \Lambda_1& 0&0\\\hline 0&\ddots&0\\\hline 0&0&\Lambda_g
                                                 \end{array}\right),
\]
where the $\Lambda_i \subset M_{r_{\fP_i}}(\Delta_\nu)$. Since
$\O_\nu$ (as scalar matrices) and $\O_L$ are contained in $\Gamma$,
Lemma~\ref{lemma:IntegralClosure} gives us that $S$, the integral
closure of $R = (\O_\nu\cap K)$ in $L$, is contained in $\Gamma$.
Thus $\O_\nu \otimes_R S \subset\O_\nu \otimes_R \Gamma = \Gamma$.  By
Proposition II.4 of \cite{Serre-LocalFields},
$\O_\nu \otimes_R S\cong \oplus_{i=1}^g \O_{\fP_i}$, so from
$\O_\nu\otimes_R S \subset \Gamma$ , we may assume that
$\O_{\fP_1} \hookrightarrow \Lambda_1$, from which we shall derive a
contradiction.

So we focus on $\Lambda_1$, the upper $r_{\fP_1}\times r_{\fP_1}$
block of
$\oplus_{i=1}^g M_{r_{\fP_i}}(\Delta_\nu) \cap \Lambda(a_1, \dots,
a_{r_\nu})$.  That intersection is contained in
\begin{equation}
\label{eq:blockmatrix}
\renewcommand\arraystretch{1.5}
\Gamma_1 \colonequals \left(\begin{array}{c|c}
M_{r_0}(\Delta_\nu)& M_{r_0\times r_{\fP_1}-r_0}(\Delta_\nu)\\\hline
\bpi M_{r_{\fP_1}-r_0\times r_0}(\Delta_\nu)& M_{r_{\fP_1}-r_0}(\Delta_\nu)\\
\end{array}\right).
\end{equation}

Write $\fP$ for $\fP_1$.  As in
Proposition~\ref{prop:unramifiedlocalextensions} and the discussion
which immediately precedes it, we write $L_\fP = K_\nu(\omega)$
($\O_\fP = \O_\nu[\omega]$) where $\omega$ is an appropriate primitive
root of unity over $K_\nu$, and $h$ is its minimal polynomial over
$K_\nu$.  We know that $h \in \O_\nu[x]$ is monic and irreducible of
degree $[L_\fP: K_\nu] = r_\fP m_\nu$.  Under the embedding
$\O_\fP \hookrightarrow \Gamma_1$ we send
$\omega \mapsto \gamma \in \Gamma_1$.  In particular, $h(\gamma) = 0$.

\textbf{Case 1:} $m_\nu=1$ ($\nu$ splits in $B$), which means
$D_\nu = K_\nu$, $\Delta_\nu = \O_\nu$, and $\fp = \pi\O_\nu$.  Let
$\chi_\gamma = \det(xI - \gamma)$ denote the characteristic polynomial
of $\gamma \in \Gamma_1 \subset M_{r_\fP}(\O_\nu)$.  Since
$\deg(\chi_\gamma) = r_\fP = \deg(h)$ and $\chi_\gamma(\gamma) = 0$,
and $h$ is irreducible, we have $h \mid \chi_\gamma$, hence
$h = \chi_\gamma$ by comparing degrees.  On the other hand viewing
$\chi_\gamma \pmod {\pi\O_\nu}$ means computing the characteristic
polynomial in
$\Gamma_1 \pmod {\pi\O_\nu} \subset M_{r_\fP}(\overline\O_\nu)$, whose
block structure will make $\chi_\gamma$ reducible mod $\pi\O_\nu$. If
$\overline h = \overline \chi_\gamma = \bar h_1 \bar h_2$ with
$\gcd(\bar h_1, \bar h_2) = 1$, then we get a nontrivial factorization
of $h$ over $\O_\nu$ by Hensel's lemma, a contradiction to the
irreducibility of $h$.  If not, then $\overline h = (\overline h_0)^k$
for some irreducible $h_0 \in \overline \O_\nu[x]$ with
$\deg(\overline h_0) < \deg(\overline h)$.  But this means that
$\overline h$ has multiple roots, contrary to
Proposition~\ref{prop:unramifiedlocalextensions}.

\textbf{Case 2:} $m_\nu > 1$.  Now $\deg(h) = r_\fP m_\nu$, and
$\gamma \in \Gamma_1 \subset M_{r_\fP}(\Delta_\nu)$.  As above, let
$W_\nu$ be a maximal unramified extension of $K_\nu$ contained in
$L_\fP \cap D_\nu$; recall $[W_\nu:K_\nu] = m_\nu$.  As a maximal
subfield of $D_\nu$, $W_\nu$ is a splitting field for $D_\nu$ and we
consider $1\otimes \gamma \in M_{r_\fP\cdot m_\nu}(W_\nu)$.  By
Theorem 9.3 of \cite{Reiner-book} the characteristic polynomial
$\chi_{1\otimes \gamma} \in \O_\nu[x]$, which is to say it is
independent of the splitting field for $D_\nu$.  As in the previous
case, we deduce that $h = \chi_{1\otimes \gamma}$.  To maintain the
flow of this argument, we defer the proof of the following lemma to
the end of this proof.

\begin{lemma}\label{lemma:chi-bar-is-reducible}
  $\overline\chi_{1\otimes\gamma}$ is reducible in
  $\overline\O_{W_\nu}[x]$.  In particular,
  $\overline\chi_{1\otimes\gamma} = \overline h_1\overline h_2$ with
  $\overline h_i \in \overline\O_{W_\nu}[x]$ and
  $\deg(\overline h_1) = r_0 < r_\fP$.
\end{lemma}
If $\overline h = \overline\chi_{1\otimes\gamma} = (\overline h_0)^k$
with $\deg(\overline h_0) < \deg(\overline h)$, then as in the
previous case $\overline h$ has multiple roots, a contradiction.  On
the other hand, if $\overline h$ factors into relatively prime
factors, Hensel's lemma will only provide a nontrivial factorization
over $\O_{W_\nu}$ which is actually expected since $h$ is irreducible
over $K_\nu$ and $[W_\nu:K_\nu] = m_\nu > 1$. So we need to dig a bit
deeper.  Let $G = \Gal(L_\fP/K_\nu)$ and $H = \Gal(L_\fP/W_\nu)$.
Then
\[h = \textrm{min}_{K_\nu}(\omega) = \prod_{\sigma \in G} (x -
\sigma(\omega)) = \prod_{\sigma\in G/H} \prod_{\tau \in H} (x -
\tau\sigma(w)).
\]
Let $h_\sigma = \prod_{\tau \in H} (x - \tau\sigma(w)).$ Since
$h_\sigma^\tau = h_\sigma$ for all $\tau \in H$, by Galois theory we
have that $h_\sigma \in \O_{W_\nu}[x]$, and
$\deg(h_\sigma) = |H| = [L_\fP:W_\nu] = r_\fP$.  Moreover since
$L_\fP = K_\nu(\omega) = K_\nu(\sigma(w))$ for any $\sigma \in G$,
$[L_\fP:W_\nu] = \deg(\min_{W_\nu}(\sigma(\omega))$, we see that
$h_\sigma = \min_{W_\nu}(\sigma(\omega))$, and so in particular,
$h = \prod_{\sigma\in G/H} h_\sigma$ is the irreducible factorization
of $h$ in $\O_{W_\nu}[x]$.

Now consider
$\overline h \in \overline\O_\nu[x] \subset \overline\O_{W_\nu}[x]$.
We have that $\overline h = \prod_{\sigma \in G/H} \overline h_\sigma$
and $\overline h_\sigma \in \overline\O_{W_\nu}[x]$. Recall that
$\overline h = \min_{\overline\O_\nu}(\overline \omega)$ and the
isomorphisms
$G = \Gal(L_\fP/K_\nu) \cong \Gal(\overline\O_\fP/\overline\O_\nu)$
and
$H = \Gal(L_\fP/W_\nu) \cong
\Gal(\overline\O_\fP/\overline\O_{W_\nu})$
give that the decomposition
$\overline h = \prod_{\sigma\in G/H} \overline h_\sigma$ is the
irreducible factorization of $\overline h$ in
$\overline\O_{W_\nu}[x]$.  But this contradicts
Lemma~\ref{lemma:chi-bar-is-reducible} which says that
$\overline h = \overline\chi_{1\otimes \gamma}$ has a factor of degree
$s < r_\fP$.
\end{proof}

\begin{proof}[Proof of Lemma~\ref{lemma:chi-bar-is-reducible}]To set
  the notation, we have $\Gamma_1 \subset M_{r_\fP}(\Delta_\nu)$.
  Following \S14 of \cite{Reiner-book}, we can choose
  $\bpi \in \Delta_\nu$ a uniformizer with $\bpi^{m_\nu} = \pi_\nu$
  ($\pi_\nu$ a uniformizer in $K_\nu$), and let $\omega_0$ be a
  primitive $q^{m_\nu}-1$ root of unity over $K_\nu$,
  $q = |\overline\O_\nu|$. So $W_\nu = K_\nu(\omega_0)$ is an
  unramified extension of $K_\nu$ in $D_\nu$ with degree $m_\nu$ over
  $K_\nu$. Then
  \[\Delta_\nu = \bigoplus_{i,j=0}^{m_\nu -1} \O_\nu \omega_0^i\bpi^j
  = \O_\nu[\omega_0,\bpi]; \qquad D_\nu = K_\nu[\omega_0, \bpi].
  \]
  In (14.6) \cite{Reiner-book}, Reiner gives an explicit
  $K_\nu$-isomorphism
  \[D_\nu \to M_{m_\nu}(W_\nu)\cong W_\nu\otimes_{K_\nu} D_\nu \mbox{
    denoted simply } a \mapsto a^*.
  \]
  From (14.7) \cite{Reiner-book}, we see that for $a \in \Delta_\nu$,
  $a^* \in M_{m_\nu}(\O_{W_\nu})$ has upper triangular image in
  $M_{m_\nu}(\overline\O_{W_\nu})$, and for $a \in \bpi \Delta_\nu$,
  $a^*$ has strictly upper triangular image in
  $M_{m_\nu}(\overline\O_{W_\nu})$.  The map $a \mapsto a^*$ now
  extends linearly to
  $M_{r_\fP}(D_\nu) \to M_{r_\fP\cdot m_\nu}(W_\nu)$.

  We first work through a simple, but non-trivial example which will
  make the general proof much easier to understand.

\begin{example}
  Let $r_0=3$, $m_\nu = 2$, and $r_\fP > r_0$ (the exact value will
  not matter).  Then

\[\renewcommand\arraystretch{1.5}\gamma \in \Gamma_1 =
\left(\begin{array}{c|c} M_3(\Delta_\nu)&M_{3\times
      r_\fp-3}(\Delta_\nu)\\\hline \bpi M_{r_\fP-3\times
      3}(\Delta_\nu)&M_{r_\fp-3}(\Delta_\nu)
      \end{array}\right).
    \]

    Then $\overline\chi_{1\otimes \gamma} = \det(-A)$ (the minus is
    for easier typesetting), where

    \newcolumntype{I}{!{\vrule width 2pt}} \newlength\savedwidth
    \newcommand\whline{\noalign{\global\savedwidth\arrayrulewidth
        \global\arrayrulewidth 2pt}%
      \hline \noalign{\global\arrayrulewidth\savedwidth}}
    \begin{equation}
      \label{eq:blockmatrix}
      A = \left[
        \begin{array}[c]{cc|cc|ccIcc|cc}
          a_{11}-x& a_{12}&a_{13}&a_{14}&a_{15}&a_{16}&a_{17}&a_{18}&\dots&\dots\\
          0&a_{22}-x&0&a_{24}&0&a_{26}&0&a_{28}&\dots&\dots\\\hline
          a_{31}&a_{32}&a_{33}-x&a_{34}&a_{35}&a_{36}&a_{37}&a_{38}&\dots&\dots\\
          0&a_{42}&0&a_{44}-x&0&a_{46}&0&a_{48}&\dots&\dots\\\hline
          a_{51}&a_{52}&a_{53}&a_{54}&a_{55}-x&a_{56}&a_{57}&a_{58}&\dots&\dots\\
          0&a_{62}&0&a_{64}&0&a_{66}-x&0&a_{68}&\dots&\dots\\\whline
          0&*&0&*&0&*&*&*&*&*\\
          0&0&0&0&0&0&*&*&*&*\\\hline
          0&*&0&*&0&*&*&*&*&*\\
          0&0&0&0&0&0&*&*&*&*\\\hline
          \vdots&\vdots&\vdots&\vdots&\vdots&\vdots&\vdots&\vdots&\vdots&\vdots\\\hline
          0&*&0&*&0&*&*&*&*&*\\
          0&0&0&0&0&0&*&*&*&*\\
        \end{array}
      \right].
    \end{equation}

    We are going to compute this determinant using minors with
    expansions focusing on columns 1, 3, 5 where the entries in the
    lower left blocks are all zero.  By expanding, we find that after
    three iterations, all of the summands in the determinant will be
    contain the determinant of the same $(r_\fp-3)\times (r_\fp-3)$
    minor. Collecting the other terms gives a degree $r_0=3$ factor.
    The notation we shall use is $A(i_1, \dots, i_r| j_1, \dots j_s)$
    will denote the matrix obtained from $A$ be removing rows
    $i_1, i_2, \dots, i_r$ and columns $j_1, j_2, \dots, j_s$.

    Expanding along the first column, we obtain:
    \[\det(A) = (a_{11}-x)\det A(1|1) + a_{31}\det A(3|1) + a_{51}\det
    A(5|1).
    \]

    In computing $\det A(m|1)$, we now look at what would be column 3
    of the original matrix $A$ which now has only two non-zero entries
    in that column of the minor.

\begin{align*}
  \det A(1|1) & = (a_{33}-x)\det A(1,3|1,3) + a_{53}\det A(1,5|1,3),\\
  \det A(3|1) &= a_{13}\det A(1,3|1,3) -a_{53}\det A(3,5|1,3),\\
  \det A(5|1) &= -a_{13}\det A(1,3|1,3) -(a_{33}-x)\det A(3,5|1,3).
\end{align*}

In this last stage we need to compute the determinant of three minors,
and the expression for each will be a multiple of
$\det A(1,3,5|1,3,5)$ from which we will obtain the claim.

\begin{align*}
  \det A(1,3|1,3) &= (a_{55}-x)\det A(1,3,5|1,3,5),\\
  \det A(1,5|1,3) &= -a_{35}\det A(1,3,5|1,3,5),\\
  \det A(3,5|1,3) &= a_{15}\det A(1,3,5|1,3,5).
\end{align*}

Now by inspection we see that we obtain a product of a cubic and a
factor of degree $r_\fP -3$. \qed
\end{example}

We now turn to the general case.  We have a
\begin{equation*}
  \renewcommand\arraystretch{1.5}
  \gamma \in \Gamma_1 \colonequals 
\left(\begin{array}{c|c}
M_{r_0}(\Delta_\nu)& M_{r_0\times r_{\fP_1}-r_0}(\Delta_\nu)\\\hline
\bpi M_{r_{\fP_1}-r_0\times r_0}(\Delta_\nu)& M_{r_{\fP_1}-r_0}(\Delta_\nu)\\
\end{array}\right).
\end{equation*}
                                     
Then
$1\otimes \gamma \in W_\nu\otimes_{K_\nu} \Gamma_1 \subset M_{r_\fP
  m_\nu}(\O_{W_\nu})$,
with reduced characteristic polynomial
$\chi_{1\otimes \gamma} \in \O_{W_\nu}[x]$ of degree $r_\fP m_\nu$.
Then the reduction, $\overline\chi_{1\otimes\gamma}$, of the
characteristic polynomial modulo $\pi_\nu \O_{W_\nu}$ is given as in
the example above as $\overline\chi_{1\otimes \gamma} = \det(-A)$, where $A$ has entries in
$\overline\O_{W_\nu}[x]$ and is given by (using $s$ for $r_0-1$, $m$
for $m_\nu$ and writing $a_{i,j}$ instead of $a_{ij}$ for clarity)

\newcolumntype{I}{!{\vrule width 2pt}}
% \newlength\savedwidth %already defined
\newcommand\whline{\noalign{\global\savedwidth\arrayrulewidth
\global\arrayrulewidth 2pt}%\hline
\noalign{\global\arrayrulewidth\savedwidth}}
{\tiny
\begin{equation}
\label{eq:blockmatrix}
\left[
\begin{array}[c]{ccc|ccc|c|cccIccc}
a_{1,1}-x& \dots&a_{1,m}&a_{1,m+1}&\dots&a_{1,2m}&\dots&a_{1,sm+1}&\dots&a_{1,r_0m}&a_{1,r_0m+1}&\dots\\
0&\ddots&\vdots&0&\ddots&\vdots&\dots&0&\ddots&\vdots&\vdots\\
0&\dots&a_{m,m}-x&0&\dots&a_{m,2m}&\dots&0&\dots&a_{m,r_0m}&\vdots\\\hline
a_{m+1,1}&\dots&a_{m+1,m}&a_{m+1,m+1}-x&\dots&a_{m+1,2m}&\dots&a_{m+1,sm+1}&\dots&a_{m+1,r_0m}&\vdots\\
0&\ddots&\vdots&0&\ddots&\vdots&\dots&0&\ddots&\vdots&\vdots\\
0&\dots&a_{2m,m}&0&\dots&a_{2m,2m}-x&\dots&0&\dots&a_{2m,r_0m}&\vdots\\\hline
\vdots&\vdots&\vdots&\vdots&\vdots&\vdots&\ddots&\vdots&\vdots&\vdots&\vdots\\\hline
a_{sm+1,1}&\dots&a_{sm+1,m}&a_{sm+1,m+1}&\dots&a_{sm+1,2m}&\dots&a_{sm+1,sm+1}-x&\dots&a_{sm+1,r_0m}&\vdots\\
0&\ddots&\vdots&0&\ddots&\vdots&\dots&0&\ddots&\vdots&\vdots\\
0&\dots&a_{r_0m,m}&0&\dots&a_{r_0m,2m}&\dots&0&\dots&a_{r_0m,r_0m}-x&\vdots\\\whline
0&*&*&0&*&*&\dots&0&*&*&a_{r_0m+1,r_0m+1}-x\\
\vdots&\ddots&*&\vdots&\ddots&*&\dots&0&\ddots&*&0\\
0&\dots&0&0&\dots&0&\dots&0&\dots&0&0&\ddots\\\hline
\vdots&\vdots&\vdots&\vdots&\vdots&\vdots&\ddots&\vdots&\vdots&\vdots&\vdots\\\hline
0&*&*&0&*&*&\dots&0&*&*&*\\
\vdots&\ddots&*&\vdots&\ddots&*&\dots&0&\ddots&*&0\\
0&\dots&0&0&\dots&0&\dots&0&\dots&0&0&\ddots\\
\end{array}
\right].
\end{equation}
}

We are going to partially compute this determinant, taking advantage
of the zeros in columns $km_\nu+1$, $k = 0, \dots, (r_0-1)$ (below row
$r_0m_\mu$).  The goal is to indicate that after $r_0$ iterations,
every minor will have the same form, and the determinant of this minor
will be therefore be a factor of the reduced characteristic polynomial
(viewed over the residue field).

Computing the determinant by expanding along the first column, we
obtain (still using $s = r_0-1$, $m$ for $m_\nu$, and writing
$a_{i,j}$ for $a_{ij}$ for clarity):
\[\det(A) = (a_{1,1}-x)\det A(1|1) + \sum_{k=1}^s a_{km+1,1}\det A(km + 1|1)\]

So at this stage our determinant involves the determinants of new
minors of the form $\det A(km+1|1)$, $k = 0, \dots, s$, that is over
column 1 and all the rows with nontrivial entries.

In computing each term $\det A(*|1)$, we next want to expand along
what would be column $m+1$ of the original matrix $A$ which now has
only $r_0-1$ non-zero entries in that column of the minor.  The final
simplification we make is that we shall not fuss about the correct
signs of each summand in the expression of the determinant since they
will be immaterial in the end, so we simply denote all of them
as~$\pm$.

\begin{align*}
  \det A(1|1) & = \pm(a_{m+1,m+1}-x)\det A(1,m+1|1,m+1) + \sum_{k=2}^s
                \pm a_{km+1,m+1}\det A(1,km+1|1,m+1).\\
  \det A(m+1|1) &= \pm a_{1,m+1}\det A(1,m+1|1,m+1) \pm
                  a_{2m+1,m+1}\det A(m+1, 2m+1|1,m+1) \pm \cdots\\
              &\quad \pm a_{sm+1,m+1}\det A(m+1,sm+1|1,m+1)\\
              &= \sum_{\substack{k=0\\k \ne 1}}^s \pm a_{km+1,m+1}\det A(km+1,m+1|
  1,m+1).\\
\end{align*}
\begin{align*}
  \det A(2m+1|1) &= \pm a_{1,m+1}\det A(1,2m+1|1,m+1) \pm (a_{m+1,m+1}-x)\det A(m+1,2m+1|1,m+1)\\
                 &\quad \pm a_{3m+1,m+1}\det A(2m+1,3m+1|1,m+1)\pm
                   \cdots\\
                 &\quad \pm a_{sm+1,m+1}\det A(2m+1,sm+1|1,m+1)\\
                 &=  \sum_{\substack{k=0\\k \ne 2}}^s \pm a_{km+1,m+1}\det
  A(2m+1,km+1|1,m+1) \mp x\det A(m+1,2m+1|1,m+1)\\
                 &\vdots\\
\end{align*}
\begin{align*}
  \det A(sm+1|1) &=\pm a_{1,m+1}\det A(1,sm+1|1,m+1) \pm
                   (a_{m+1,m+1}-x)\det A(m+1,sm+1|1,m+1)\\
                 &\quad \pm a_{2m+1,m+1}\det A(2m+1,sm+1|1,m+1)\pm
                   \cdots\\
                 &\quad \pm a_{(s-1)m+1,m+1}\det A((s-1)m+1,sm+1|1,m+1)\\
                 &=  \sum_{\substack{k=0\\k \ne s}}^s \pm a_{km+1,m+1}\det
  A(sm+1,km+1|1,m+1) \mp x\det A(m+1,sm+1|1,m+1)\\
\end{align*}

We need to take stock of what is happening.  Each of these minors has
the form $A(*|1, m+1)$.  It is clear and we continue to evaluate the
determinants of these minors, the next set will have the form
$A(*|1, m+1, 2m+1)$ and after $r_0$ iterations will have the form
$A(*|1, m+1, 2m+1, \dots, sm+1)$.

Also at our current stage of computation, all minors of the form
$A(jm+1,k m+1|1, m+1)$ where $j \ne k \in \{0, \dots, s\}$ also occur.
At each new stage a new row will be added to the minor
$jm+1, km+1, \ell m+1$ where $j,k,l$ range over 0, \dots, $s$ with all
indices distinct.  After $r_0$ iterations, all $r_0$ rows $km+1$,
$k=0, \dots, s$ will necessarily appear in each minor, at which point
we will have
\[\overline \chi_{1\otimes \gamma} = \overline h_1 \cdot \det A(1,
m+1, \dots, sm+1|1, m+1, \dots, sm+1).
\]
Moreover, if $A_0$ was the image of the matrix of $1\otimes \gamma$ in
$M_{r_\fP m_\nu}(\overline \O_{W_\nu})$ (so that the matrix $A$ above
is $A =\det(xI - A_0)$), we would have that
\[\det - A(1, m+1, \dots, sm+1|1, m+1, \dots, sm+1) = det(xI - A_0(1,
\dots, sm+1|1, \dots, sm+1)),
\]
that is the characteristic polynomial of a matrix in
$M_{r_\fP m_\nu - r_0}(\O_{W_\nu})$, and thus having degree
$r_\fp m_\nu - r_0$.  This establishes that
$\overline \chi_{1\otimes \gamma} = \overline h_1 \overline h_2$ where
$\deg \overline h_1 = r_0 < r_\fP$, which completes the proof.
\end{proof}

\section{Constructing Distinguished Representatives of the
  Isomorphism Classes of Maximal Orders}
The goal of this section is to use
the local result (Theorem~\ref{thm:OLembedslocally}) and a
local-global principle to construct a set of representatives of the
isomorphism classes of maximal orders in $B$, and distinguish those
which are guaranteed to contain $\O_L$.  This task involves a number
of steps.  The first is to define a class field $K(\R)/K$ whose degree
is the number of isomorphism classes comprising the genus of $\R$.
Then places $\nu$ of $K$ are chosen so that the Artin symbols
$(\nu, K(\R)/K)$ correspond to generators of $\Gal(K(\R)/K)$ in which
$\nu$ has prescribed splitting behavior in $L$.  Finally, a set of
maximal orders in $B$ are constructed by choosing distinguished
representatives of the local algebras $B_\nu$ using
Theorem~\ref{thm:OLembedslocally}. This broad outline was also
followed in the simpler case of prime degree
\cite{Linowitz-Shemanske-EmbeddingPrimeDegree}, but we include all the
details here to afford careful treatment especially to the
complications which arise due to the presence of partial ramification
for central simple algebras of arbitrary degree.

\subsection{Class fields and the genus of $\R$}

First, we construct a class field, $K(\R)$, associated to the maximal
order $\R$ whose degree over $K$ equals the number of isomorphism
classes of maximal orders in the global algebra $B$.  We then we give
a filtration of the Galois group, $\Gal(K(\R)/K)$, in order to
parametrize the isomorphism classes of maximal orders in $B$.

The class field extension $K(\R)/K$ comes from class field theory by
producing an open subgroup $H_\R$ of finite index in the idele group
$J_K$.  The group $H_\R$ is the product of $K^\times$ and the reduced
norm of an idelic normalizer of $\R$ ($nr(\fN(\R))$, where
$\fN(\R) = J_B \cap \prod_\nu \N(\R_\nu)$, and where $\N(\R_\nu)$ is
the local normalizer of $\R_\nu$ in $B_\nu^\times$, and $J_B$ is the
idele group of $B$.)  We begin by computing the local normalizers and
their reduced norms.

\subsubsection{Normalizers and their reduced
  norms.}\label{sec:localnormalizers}

 Given our maximal order $\R \subset B$ and a place $\nu$
of $K$, we have previously defined the completions $\R_\nu \subseteq B_\nu$.
Let $\N(\R_\nu)$ denote the normalizer of $\R_\nu$ in $B_\nu^\times$,
and $nr_{B_\nu/K_\nu}(\N(\R_\nu))$ its reduced norm in $K_\nu^\times$.
First suppose that $\nu$ is an infinite place, so
$\N(\R_\nu) = B_\nu^\times$.  If $\nu$ splits in $B$, then
$B_\nu\cong \Mat_n(K_\nu)$, so $\N(R_\nu) \cong GL_n(K_\nu)$, and
$nr_{B_\nu/K_\nu}(\N(\R_\nu)) = K_\nu^\times$, while if $\nu$ ramifies
in $B$ (possible only if $n$ is even and $\nu$ is real), then (33.4)
of \cite{Reiner-book} shows that
$nr_{B_\nu/K_\nu}(\N(\R_\nu)) = \mathbb R_+^\times$.

 For a finite place $\nu$, it is clearest to distinguish
three cases.  If $m_\nu = 1$ (the split case), then $B_\nu$ has been
identified with $\Mat_n(K_\nu)$, so by (17.3) and (37.26) of
\cite{Reiner-book}, every maximal order is conjugate by an element of
$B_\nu^\times$ to $\Mat_n(\O_\nu)$, and every normalizer is conjugate
to $\GL_n(\O_\nu) K_\nu^\times$, hence
$nr_{B_\nu/K_\nu}(\N(\R_\nu)) =\O_\nu^\times (K_\nu^\times)^n$.

At the other extreme is $m_\nu = n$ (the totally ramified case), so
that $B_\nu = D_\nu$.  Then $\R_\nu$ is the unique maximal order of
the division algebra $B_\nu$, so $\N(\R_\nu) = B_\nu^\times$, and by
p~153 of \cite{Reiner-book},
$nr(\N(\R_\nu)) = nr_{B_\nu/K_\nu}(B_\nu^\times) = K_\nu^\times.$

Finally, consider the partially ramified case in which
$B_\nu \cong \Mat_{r_\nu}(D_\nu)$ where $D_\nu$ is a central division
algebra of degree $1<m_\nu<n$ over $K_\nu$.  Then $\R_\nu$ is
conjugate to $\Mat_{r_\nu}(\Delta_\nu)$ where $\Delta_\nu$ is the
unique maximal order of $D_\nu$ (17.3 of \cite{Reiner-book}).

From \S 14.5 of \cite{Reiner-book}, we choose a uniformizer
$\bpi = \bpi_{D_\nu}$ for $\Delta_\nu$ so that
$\bpi^{m_\nu} = \pi_\nu \in K_\nu$.  We also take $\omega$ a primitive
$(q^{m_\nu}-1)$th root of unity in $\Delta_\nu$.  Then
$E_\nu = K_\nu(\bpi)$ and $W_\nu = K_\nu(\omega)$ are degree $m_\nu$ field
extensions of $K_\nu$ which are respectively totally ramified and
unramifed and so that
\begin{equation}
  \label{eqn:Deltadecomp}
  \Delta_\nu = \O_\nu[\omega,\bpi] = \bigoplus_{i,j=0}^{m_\nu-1}\O_\nu \omega^i \bpi^j
  \mbox{ and }D_\nu = K_\nu[\omega, \bpi].
\end{equation}

To deduce $nr(\N(\R_\nu))$, it is sufficient to consider
$\R_\nu = M_{r_\nu}(\Delta_\nu)$.  From (37.25)-(37.27) of
\cite{Reiner-book}, we know that
$\N(\R_\nu)/GL_{r_\nu}(\Delta_\nu)K_\nu^\times \cong \Z/m_\nu\Z$.  By
(17.3) of \cite{Reiner-book}, we know that $\bpi \R_\nu$ is the unique
two-sided ideal of $\R_\nu$, which is to say that
$\bpi \in \N(\R_\nu)$.  It follows that $\N(\R_\nu)$ is the group
generated by $\bpi I_{r_\nu}$ and
$GL_{r_\nu}(\Delta_\nu)K_\nu^\times $.  Since
$nr_{D_\nu/K_\nu}(\bpi) = (-1)^{m_\nu-1}\pi_\nu$, we have
$nr_{B_\nu/K_\nu}(\bpi I_{r_\nu})
=(-1)^{r_\nu(m_\nu-1)}\pi_\nu^{r_\nu}.$
Finally, given that the unramified extension $W_\nu/K_\nu$ is
contained in $\Delta_\nu$ and and the norm $N_{W_\nu/K_\nu}$ maps the
units of $\O_{W_\nu}$ onto $\O_\nu^\times$, we may conclude that
$nr(\N(\R_\nu)) = \O_\nu^\times (K_\nu^\times)^{r_\nu}$.

Summarizing, for a finite place $\nu$ of $K$, the computations above
show that
\[ nr(\N(\R_\nu)) = nr_{B_\nu/K_\nu}(\N(\R_\nu))= \O_\nu^\times
(K_\nu^\times)^{r_\nu}, \] for all $1 \le r_\nu \le n$.

Thus with the exception of a real place $\nu$ which ramifies in $B$
(possible only if $n$ is even), for all places $\nu$ we have
$\O_\nu^\times \subset nr(\N(\R_\nu))$, a fact that will be important
in associating a class field to $\R$.

\subsubsection{Parametrizing the Genus}

We know that any two maximal orders in $B$ are locally conjugate at
all (finite) places of $K$, so the number of isomorphism classes can
be computed adelically as follows.  Let $J_B$ be the idele group of
$B$, and let $\fN(\R) = J_B \cap \prod_\nu \N(\R_\nu)$ be the adelic
normalizer of $\R$.  The number of isomorphism classes of maximal
orders is the cardinality of the double coset space
$B^\times\backslash J_B/\fN(\R)$.  To make use of class field theory,
we need to realize this quotient in terms of the arithmetic of
$K$. The reduced norms on the local algebras $B_\nu$ induce a natural
map $nr: J_B \to J_K$, where $J_K$ is the idele group of $K$, and
where for $\tilde \alpha = (\alpha_\nu)_\nu \in J_B$,
$nr(\tilde\alpha) \colonequals (nr_{B_\nu/K_\nu}(\alpha_\nu))_\nu$.

The theorem below was proven (Theorem 3.1 of
\cite{Linowitz-Shemanske-EmbeddingPrimeDegree}) for $\deg_KB = p$ an
odd prime.  The changes required for general degree $n$ involve
handling possible ramification at an infinite place, and pervade the
proof, so we repeat the full argument in the interest of clarity.

\begin{theorem}\label{thm:type_number_bijection}Let $n = \deg_KB \ge
  3$. The reduced norm induces a bijection
  \[
  nr: B^\times\backslash J_B/\fN(\R) \to K^\times \backslash
  J_K/nr(\fN(\R)).
  \]
  The group $K^\times \backslash J_K/nr(\fN(\R))$ is abelian with
  exponent~$n$.
\end{theorem}

\begin{remark}
  The proof below is valid for $n=2$ as well as long as $B$ satisfies
  the Eichler condition.  The map is always surjective, but
  injectivity requires strong approximation.
\end{remark}

\begin{proof}The map is defined in the obvious way with
  $nr(B^\times \tilde \alpha \fN(\R)) = K^\times nr(\tilde
  \alpha)\,nr(\fN(\R))$,

  We first show the mapping is surjective.  Let
  $\tilde a = (a_\nu)_\nu \in J_K$ and
  $K^\times \tilde a\, nr(\fN(\R))$ be the associated double coset in
  $K^\times \backslash J_K/nr(\fN(\R))$.  The weak approximation
  theorem implies the existence of an element $c \in K^\times$ so that
  $c \tilde a$ satisfies $ca_\nu > 0$ for all real places $\nu$ of $K$
  which ramify in $B$ (if any).  Since (replacing $a$ by $ca$) the
  associated double cosets are equal, we may assume without loss that
  $\tilde a$ was chosen with $a_\nu > 0$ are all the real places which
  ramify in $B$.

  Now we appeal to (33.4) of \cite{Reiner-book} which says that for
  any place $\nu$ of $K$, $nr_{B_\nu/K_\nu}(B_\nu) = K_\nu$ with the
  sole exception of $K_\nu \cong \mathbb R$ and $B$ ramified at $\nu$
  in which case the image of the norm is the non-negative reals.  Let
  $S$ be a finite set of places of $K$ containing all the archimedean
  places and all places which ramify in $B$.  By (33.4) and the
  assumptions on $\tilde a$ at the real places, for each place
  $\nu \in S$, there exists $\beta_\nu \in B_\nu^\times$ so that
  $nr_{B_\nu/K_\nu}(\beta_\nu) = a_\nu$.

  Now let $\nu$ be a place of $K$, with $\nu \notin S$. We have that
  $\R_\nu$ is conjugate to $M_n(\O_{K_\nu})$, so let
  $\beta_\nu \in \R_\nu$ be conjugate to
  $\diag(a_\nu, 1,\dots,1) \in M_n(\O_{K_\nu})$.  Then
  $nr_{B_\nu/K_\nu}(\beta_\nu) = nr_{B_\nu/K_\nu}(\diag(a_\nu,
  1,\dots,1)) = a_\nu.$
  So now put $\tilde \beta = (\beta_\nu)_\nu$. It is clear that
  $\tilde \beta = (\beta_\nu)_\nu \in J_B$ and
  $nr_{J_B/J_K}(\tilde \beta) = \tilde a$, which establishes
  surjectivity.

  To prove injectivity, we first prove a claim: The preimage of
  $K^\times nr(\fN(\R))$ under $nr$ is $B^\times J^1_B\fN(\R)$ where
  $J_B^1$ is the kernel of the norm map: $nr:J_B \to J_K$.  It is
  obvious that
  $nr(B^\times J^1_B\fN(\R)) \subset K^\times nr(\fN(\R))$.  Let
  $\tilde \gamma = (\gamma_\nu)_\nu \in J_B$ be such that
  $nr(B^\times \tilde \gamma\fN(\R)) \in K^\times nr(\fN(\R)).$ Then
  $nr(\tilde \gamma) \in K^\times nr(\fN(\R))$, so write
  $nr(\tilde \gamma) = a\cdot nr(\tilde r)$ where $a \in K^\times$ and
  $\tilde r = (r_\nu)_\nu \in \fN(\R)$. We claim that $a$ is positive
  at all the real places which ramify in $B$.  Indeed writing $a_\nu$
  for the image of $a$ under the embedding
  $K \subset K_\nu \cong \mathbb R$, we have that
  $nr_{B_\nu/K_\nu}(\gamma_\nu) = a_\nu nr_{B_\nu/K_\nu}(r_\nu)$, with
  $nr_{B_\nu/K_\nu}(\gamma_\nu), nr_{B_\nu/K_\nu}(r_\nu)>0$.  It
  follows by the Hasse-Schilling-Maass theorem (Theorem 33.15 of
  \cite{Reiner-book}) that there is an element $b \in B^\times$ so
  that $nr_{B/K}(b) = a$, and so that
  $nr(\tilde \gamma) =nr(b) nr(\tilde r)$, or
  $nr(b^{-1})nr(\tilde \gamma)nr(\tilde \gamma^{-1}) = 1 \in J_K$.
  Thus $b^{-1}\tilde \gamma\, \tilde r^{-1} \in J_B^1$, and
  $B^\times\tilde\gamma\fN(\R) = B^\times b^{-1}\tilde \gamma\, \tilde
  r^{-1}\fN(\R) \in B^\times J_B^1\fN(\R)$ as claimed.

  To proceed with the proof of injectivity, suppose that that are
  $\tilde \alpha$, $\tilde \beta \in J_B$ so that
  $nr(B^\times\tilde \alpha\, nr(\fN(\R)) = nr(B^\times\tilde \beta\,
  nr(\fN(\R))$.  Then
  \[K^\times nr(\tilde \alpha) nr(\fN(\R)) = K^\times nr(\tilde \beta)
  nr(\fN(\R)),\]
  which since $J_K$ is abelian, implies that
  $nr(\tilde\alpha^{-1}\tilde \beta) \in K^\times nr(\fN(\R))$, so by
  the above claim,
  $\tilde \alpha^{-1}\tilde \beta \in B^\times J_B^1 \fN(\R).$

  Now the subgroup $B^\times J_B^1$ is the kernel of the homomorphism
  $J_B \to J_K/K^\times$ induced by $nr$, so that
  $\tilde \beta \in B^\times J_B^1 \fN(\R) = B^\times J_B^1\tilde
  \alpha \fN(\R).$
  By VI.iii and VII of \cite{Frohlich-lf},
  $J_B^1 \subset B^\times \tilde\gamma \fN(\R)\tilde\gamma^{-1}$ for
  any $\tilde\gamma \in J_B$, so choosing
  $\tilde\gamma = \tilde\alpha$, we get
  \[\tilde \beta \in B^\times J_B^1\tilde \alpha\, \fN(\R)\subset
  B^\times (B^\times \tilde \alpha\, \fN(\R)\,\tilde\alpha^{-1})\tilde
  \alpha\, \fN(\R) = B^\times\tilde\alpha\,\fN(\R).
  \]
  Thus
  $B^\times \tilde \beta \, \fN(\R) \subseteq B^\times \tilde
  \alpha\,\fN(\R)$, and and by symmetry, we have equality.

  To see that the group has exponent $n$, we note that the local
  factors in $J_K/nr(\fN(\R))$ have the form
  $K_\nu^\times/ nr_{B_\nu/K_\nu}(\N(\R_\nu))$.  From our computations
  above, we see that for $\nu$ a finite place, this quotient is either
  trivial or  equal to
  $K_\nu^\times/(\O_\nu^\times (K_\nu^\times)^r)$ (for $r\mid n$)
  which clearly has exponent $n$,  and that if $\nu$ is
  an infinite place, the quotient is trivial unless $\nu$ is a real
  place which ramifies in $B$.  In that case,
  $K_\nu^\times/ nr(\N(\R_\nu)) = \mathbb R^\times/\mathbb R_+^\times
  \cong \Z/2\Z$,
  but in that case $n$ is necessarily even, so again the factor has
  exponent $n$.
\end{proof}

 We have seen above that the distinct isomorphism classes
of maximal orders in $B$ are in one-to-one correspondence with the
double cosets in the group
$K^\times \backslash J_K/nr(\fN(\R)) \cong G_\R := J_K/H_\R$, where
$H_\R = K^\times nr(\fN(\R))$.  Since $H_\R$ contains a neighborhood
of the identity in $J_K$, it is an open subgroup (Proposition II.6 of
\cite{Higgins}) having finite index, and so by class field theory
\cite{Lang-ANT}, there is a class field, $K(\R)$, associated to it.
The extension $K(\R)/K$ is an abelian extension with
$\Gal(K(\R)/K) \cong G_\R$.   Moreover, a place $\nu$ of
$K$ (possibly infinite) is unramified in $K(\R)$ if and only if
$\O_\nu^\times \subset H_\R$, and splits completely if and only if
$K_\nu^\times \subset H_\R$.  Here if $\nu$ is archimedean, we take
$\O_\nu^\times = K_\nu^\times$.

\begin{remark}\label{remark:classfieldunramifield-orderArtinsymbol}
  From our computations above, we see (unless there is a real place of
  $K$ which ramifies in $B$) that $\O_\nu^\times$ is always contained
  in $H_\R$.  In particular the class field $K(\R)/K$ is unramified
  outside of the real places which ramify in $B$, so contained in the
  narrow class field of $K$.

  It is also useful to make a simple observation about the order of
  Artin symbols in the class field extension $K(\R)/K$.  For a finite
  place $\nu$ of $K$ and $\pi_\nu$ a uniformizer in $K_\nu$, the
  isomorphism $G_\R = J_K/H_\R \to \Gal(K(R)/K)$ associates the image
  of the idele $\tilde\omega_\nu = (\dots, 1, \pi_\nu, 1,\dots)$ in
  $G_\R$ with the Artin symbol $(\nu, K(\R)/K)$.  Since
  $\tilde\omega_\nu^{r_\nu} = 1$ in $G_\R$ we have that the order of
  the Artin symbol (the inertial degree) $f(\nu;K(\R)/K)$ divides
  $r_\nu$.\qed
\end{remark}

Our goal in what follows is to determine a subgroup $H$ of
the Galois group $G = \Gal(K(\R)/K)$ so that each isomorphism class of
maximal order in $B$ corresponding to an element of $H$ contains a
representative which contains the ring of integers $\O_L$.  On the
other hand, the process of identifying the representatives containing
$\O_L$ requires a slightly finer filtration of the group $G$ which we
establish below.

We begin by specifying a set of generators for the group $G$ as Artin
symbols, $(\nu, K(\R)/K)$, in such a way that we can control the
splitting behavior of $\nu$ in the extension $L/K$.  As $L$ is an
arbitrary extension of $K$ of degree $n$, this requires some care.

We have assumed that $L \subset B$. Put $L_0 = K(\R) \cap L$ and
$\wL_0 = \wL \cap K(\R)$ where $\wL$ is the Galois closure of
$L$. Then $L_0 \subset \wL_0$ and we define subgroups of $G$:
$\wH = \Gal(K(\R)/\wL_0) \subseteq H = \Gal(K(\R)/L_0)$.  We write the
finite abelian groups $\wH$, $H/\wH$, and $G/H$ as a direct product of
cyclic groups:
\begin{align}
  \label{eq:param1}
  G/H = &\la\rho_1 H\ra \times \cdots \times \la \rho_r H\ra,\\
  \label{eq:param2}
  H/\wH &= \la \sigma_1 \wH\ra \times \cdots \times \la \sigma_s \wH
          \ra,\\
  \label{eq:param3}
  \wH &= \la \tau_1\ra \times \cdots \times \la \tau_t\ra.
\end{align}

The following proposition is clear.
\begin{prop}\label{prop:Galoisrepn}
  Every element $\varphi \in G$ can be written uniquely as
  $\varphi = \rho_1^{a_1} \cdots \rho_r^{a_r}\sigma_1^{b_1} \cdots
  \sigma_s^{b_s} \tau_1^{c_1}\cdots \tau_t^{c_t}$
  where $0 \le a_i < |\rho_i H|$, $0 \le b_j < |\sigma_j\wH|$, and
  $0\le c_k < |\tau_k|$, with $|\cdot|$ the order of the element in
  the respective group.
\end{prop}

Next we characterize each of these generators in terms of Artin
symbols.  Since the vehicle to accomplish this is the Chebotarev
density theorem which provides an infinite number of choices for
places, we may and do assume without loss that the places we choose to
define the Artin symbols are unramified in both $\wL/K$ and $B$.

First consider the elements $\tau_k \in \wH = \Gal(K(\R)/\wL_0)$.  By
Lemma 7.14 of \cite{Narkiewicz-book}, there exist infinitely many
places $\nu_k$ of $K$ so that $\tau_k = (\nu_k, K(\R)/K)$ and for
which there exists a place $Q_k$ of $\wL$ with inertia degree
$f(Q_k\mid \nu_k) = 1$.  Since $\wL/K$ is Galois (and the place
$\nu_k$ is unramified by assumption), this implies $\nu_k$ splits
completely in $\wL$, hence also in $L$.

Next consider $\sigma_j\wH$ with $\sigma_j \in H = \Gal(K(\R)/L_0)$.
Again by Lemma 7.14 of \cite{Narkiewicz-book}, there exist infinitely
many places $\mu_j$ of $K$ so that $\sigma_j = (\mu_j, K(\R)/K)$ and
for which there exists a place $Q_j$ of $L$ with inertia degree
$f(Q_j\mid \mu_j) = 1$. Here the $\mu_j$ need not split
completely in $L$.

Finally consider $\rho_kH$ with $\rho_k \in G = \Gal(K(\R)/K)$.  By
Chebotarev, there exist infinitely many places $\lambda_i$ of $K$ so
that $\rho_i = (\lambda_i, K(\R)/K)$.  For later convenience, we note
that by standard properties of the Artin symbol,
$\overline\rho_i = \rho_i\vert_{L_0} = (\lambda_i, L_0/K)$ whose order
in $\Gal(L_0/K)$ is equal to the inertia degree $f(\lambda_i; L_0/K)$.

As we said above, we have assumed without loss that all the places
$\lambda_i, \mu_j, \nu_k$ are unramified in $\wL$ and not totally
ramified in $B$.

\subsection{Fixing representatives of the isomorphism classes}

In the previous subsection, we have chosen generators for
$\Gal(K(\R)/K)$ which are characterized as Artin symbols, in
particular associated to certain finite places of $K$ whose splitting
behavior in our given extension $L/K$ is somewhat controlled.  We
recall that the size of the Galois group equals the number of
isomorphism classes of maximal orders in $B$.  At each of those places
$\nu$ associated to an Artin symbol, we consider the local algebra,
$B_\nu$, and specify a certain collection of maximal orders in it (the
number being equal to the order of the Artin symbol $(\nu, K(\R)/K)$),
and loosely speaking, take as many local orders as possible which
contain $\O_L$.  We will then fix representatives of the isomorphism
classes of maximal orders in $B$ by utilizing a local-global
correspondence.

As above, $\R$ is a fixed maximal order of $B$ containing $\O_L$.  For
a finite place $\nu$ of $K$ which is not totally ramified in $B$, we
have $B_\nu \cong M_{r_\nu}(D_\nu)$, with $D_\nu$ a central division
algebra over $K_\nu$ with unique maximal order $\Delta_\nu$, and
$r_\nu > 1$. We fix an apartment in the affine building for
$SL_{r_\nu}(D_\nu)$ which contains the vertex corresponding to the
maximal order $\R_\nu$. We may select a basis
$\{\alpha_1, \dots, \alpha_{r_\nu}\}$ of $D_\nu^{r_\nu}$ so that
$\R_\nu = M_{r_\nu}(\Delta_\nu) \cong \End_{\Delta_\nu}(\Lambda)$
where $\Lambda = \bigoplus_{i=1}^{r_\nu} \Delta_\nu \alpha_i$.  With
$\bpi$ a uniformizer of $\Delta_\nu$, the vertices of the apartment
are in bijective correspondence with those maximal orders of $B_\nu$
which are given as endomorphism rings of lattices of the form
$\bigoplus_{i=1}^{r_\nu} \Delta_\nu \bpi^{a_i} \alpha_i$, the
homothety class of which we abbreviate by
$[a_1, \dots, a_{r_\nu}] \in \Z^{r_\nu}/\Z(1,\dots,1)$.  We shall
identify the vertices of the apartment with these homothety classes of
lattices.

Let's understand how this applies to choosing our representatives for
the isomorphism classes.  Since $L \subset B$, we know (by the
Albert-Brauer-Hasse-Noether theorem) that $m_\nu \mid [L_\fP:K_\nu]$
for all places $\nu$ of $K$ and places $\fP$ of $L$ lying above $\nu$.
For finite places $\nu$, we have that $L_\fP$ embeds as a
$K_\nu$-algebra into $B_\nu \cong M_{r_\fP}(D_\nu)$ where
$r_\fP = [L_\fP:K_\nu]/m_\nu$ is minimal.

Corresponding to various generators of the $\Gal(K(\R)/K)$ we have
chosen finite places $\lambda_i$, $\mu_j$, and $\nu_k$ to parametrize
the Artin symbols which represent the generators. We now consider
maximal orders in the associated local algebras.  For a generic place
$\nu$ among these (which we recall can be assumed unramified in
$L/K$), let $\nu \O_L = \fP_1 \cdots \fP_g$ be the prime factorization
in $L$.  By Theorem~\ref{thm:OLembedslocally}, we know that $\O_L$ is
a subset of precisely those maximal orders (vertices of the apartment)
associated to homothety classes of lattices of the form
$[\mathcal L] = [\underbrace{\ell_1,
  \dots,\ell_1}_{r_{\fP_1}},\underbrace{\ell_2,
  \dots,\ell_2}_{r_{\fP_2}},\dots, \underbrace{\ell_g,
  \dots,\ell_g}_{r_{\fP_g}}]$, $\ell_i \in \Z$.

We will be particularly interested in maximal orders of the form
$\R(k,\ell)$ defined in equation~(\ref{eq:blockstr}).  Because we
shall vary the place $\nu$ in the parametrization below, we will write
$\R_\nu(k,\ell)$ for $\R(k,\ell)$ to make the dependence on $\nu$
explicit. Recall that $\R_\nu(k,\ell)$ corresponds to the homothety
class
$[\underbrace{\ell, \dots, \ell}_k,0, \dots, 0] \in
\Z^{r_\nu}/\Z(1,\dots,1)$ which has type $k\ell\pmod {r_\nu}$.

By equation~(\ref{eq:orderscontainingomega}),
\[ \O_L \subset \bigcap_{\ell_i \in \Z}\big[ \R_\nu(r_{\fP_1},\ell_1)
\ \cap \R_\nu(r_{\fP_1}+r_{\fP_2},\ell_2) \cap \cdots \cap
\R_\nu(r_{\fP_1}+\cdots+r_{\fP_g},\ell_g) \big] = \bigoplus_{\fP\mid
  \nu} M_{r_\fP}(\Delta_\nu) \subset M_{r_\nu}(D_\nu).
\]

Now for the places $\lambda_i$, $\mu_j$, and $\nu_k$ we specified
above to parametrize $G = \Gal(K(\R)/K)$, fix the following local
orders using the decomposition of $G$ into $G/H$, $H/\wH$, and $\wH$:

The places $\nu_k$ all split completely in $L$, so
$L_\fP = K_{\nu_k}$, and $m_\nu \mid [L_\fP:K_{\nu_k}]$ implies
$r_\fP = m_{\nu_k} = 1$, and that $r_{\nu_k} = n$.

So for each place $\nu_k$ ($k = 1, \dots, t$) whose Artin symbol
$(\nu_k, K(\R)/K)) = \tau_k$ is one of the generators of $\wH$, we fix
vertices $\R_{\nu_k}(m,1)$, $m = 0, 1, \dots, |\tau_k| -1$ with
associated homothety classes $[0, \dots, 0]$, $[1,0,\dots, 0]$,
$[1,1,0,\dots,0]$, \dots,
$[\underbrace{1,\dots,1}_{ |\tau_k| -1}, 0,\dots, 0]$. Note that since
$r_{\nu_k} = n$ and $\tau_k$ has exponent $n$, all these homothety
classes correspond to vertices in a fundamental chamber of the
building, and the corresponding maximal orders contain $\O_L$ by
equation~(\ref{eq:orderscontainingomega}).

Now consider the places $\mu_j$ ($j = 1, \dots, s$) whose Artin symbol
$(\mu_j, K(\R)/K)) = \sigma_j$ gives one of the generators
$\sigma_j\wH$ of $H/\wH$. Recall that each $\mu_j$ factors into places
of $L$ with at least one having inertia degree one over $\mu_j$, say
$\fP_1$. Since $\mu_j$ is (by choice) unramified in $L$, we have as in
the previous case $m_{\mu_j} \mid [L_{\fP_1}:K_{\mu_j}]=1$, which
forces $m_{\mu_j} = r_{\fP_1} = 1$ and
$r_{\mu_j} = r_{\mu_j}m_{\mu_j} = n$.  From
equation~(\ref{eq:orderscontainingomega}),
$\O_L \subset \R_{\mu_j}(r_{\fP_1},\ell_1) = \R_{\mu_j}(1,\ell_1)$ for
all $\ell_1 \in \Z$, so we fix vertices $\R_{\mu_j}(1,m)$,
$m = 0, 1, \dots, |\sigma_j \wH| -1$ with associated homothety classes
$[0, \dots, 0]$, $[1,0,\dots, 0]$, $[2,0,\dots,0]$, \dots,
$[|\sigma_j\wH| -1,0,\dots, 0]$ in a fundamental apartment.

Finally consider the places $\lambda_i$ ($i = 1, \dots, r$) whose
Artin symbol $(\lambda_i, K(\R)/K)) = \rho_i$ gives one of the
generators $\rho_iH$ of $G/H$.  It is only here where selectivity can
manifest itself.

Recall that via the isomorphism $G/H \cong \Gal(L_0/K)$
($\rho_iH \leftrightarrow \overline \rho_i$), we know that the order
of $\rho_i H$ is the inertia degree $f(\lambda_i; L_0/K)$ which we
have shown divides $r_{\lambda_i}$.  So we wish to specify
$f(\lambda_i; L_0/K)$ maximal orders in the local algebra.  From
Theorem~\ref{thm:OLembedslocally}, we know that $\O_L$ is contained in
maximal orders corresponding precisely to vertices whose associated
homothety classes are of the form
$[\underbrace{\ell_1, \dots,\ell_1}_{r_{\fP_1}},\underbrace{\ell_2,
  \dots,\ell_2}_{r_{\fP_2}},\dots, \underbrace{\ell_g,
  \dots,\ell_g}_{r_{\fP_g}}]$,
in particular having types
$\sum_{k=1}^g r_{\fP_k} \ell_k \pmod {r_{\lambda_i}}$.  Since the
$\ell_k$ are arbitrary integers, $\O_L$ is contained in maximal orders
having types which are multiples of
$d_{\lambda_i} = \gcd(r_{\fP_1}, \dots, r_{\fP_g})$; note that
$d_{\lambda_i}\mid r_{\lambda_i} = \sum_{k=1}^g r_{\fP_k}$.

\begin{remark}
  We need to be a bit careful in leveraging the above observation.  We
  have shown that $\O_L$ is contained in maximal orders having types a
  multiple of $d_{\lambda_i}$, but the converse is not necessarily
  true.  For example, suppose
  $r_{\lambda_i} = \sum_{k=1}^g r_{\fP_k} = 1 +2$, so that $\O_L$ is
  contained in maximal orders corresponding to homothety classes of
  the form $[\ell_1, \ell_2, \ell_2].$ Now $d_{\lambda_i} = 1$, so
  $\O_L$ is contained in maximal orders associated to homothety
  classes of all types, in particular type 1, but for example $\O_L$
  is not contained in the maximal order corresponding to the homothety
  class of the lattice $[0,1,0]$ since that is not of the prescribed
  form: $[\ell_1, \ell_2, \ell_2]$.  This presents no serious issue,
  but we need to be somewhat careful in selecting our representatives.
\end{remark}

Fix integers $\ell_1, \dots, \ell_k$ so that
\[d_{\lambda_i} = \gcd(r_{\fP_1}, \dots, r_{\fP_g})= r_{\fP_1}\ell_1 +
\cdots + r_{\fP_g}\ell_g,\]
and fix a vertex corresponding to the homothety class
\[[\L] = [\underbrace{\ell_1,
  \dots,\ell_1}_{r_{\fP_1}},\underbrace{\ell_2,
  \dots,\ell_2}_{r_{\fP_2}},\dots, \underbrace{\ell_g,
  \dots,\ell_g}_{r_{\fP_g}}].\]
Using somewhat ad hoc notation, for an integer $a$, let
\[[\L^a] = [\underbrace{a\ell_1,
  \dots,a\ell_1}_{r_{\fP_1}},\underbrace{a\ell_2,
  \dots,a\ell_2}_{r_{\fP_2}},\dots, \underbrace{a\ell_g,
  \dots,a\ell_g}_{r_{\fP_g}}],
\]
which has type $ad_{\lambda_i} \pmod {r_{\lambda_i}}.$ Now
\[ d_{\lambda_i}x \equiv d_{\lambda_i}y \pmod {r_{\lambda_i}} \mbox{
  iff } x \equiv y \pmod{r_{\lambda_i}/d_{\lambda_i}},
\]
so this process will produce $r_{\lambda_i}/d_{\lambda_i}$ maximal
orders which contain $\O_L$, representing every possible type of
maximal order which can contain $\O_L$.  It turns out that in general,
there will be some redundancy when we use these local orders to
construct global ones via a local-global correspondence.  We need to
correct for this, and we begin with an elementary claim.

\begin{lemma}
  With the notation as above except abbreviating $f(\lambda_i;L_0/K)$
  by $f_{\lambda_i}$, we have
  \[\frac{f_{\lambda_i}}{\gcd(d_{\lambda_i}, f_{\lambda_i})} \biggm|
  \frac{r_{\lambda_i}}{d_{\lambda_i}}.
  \]
\end{lemma}

\begin{proof}
  We know that $f_{\lambda_i} \mid r_{\lambda_i}$ and
  $d_{\lambda_i}\mid r_{\lambda_i}$.  Then
  \[\frac{r_{\lambda_i}}{d_{\lambda_i}} \cdot
  \frac{\gcd(d_{\lambda_i}, f_{\lambda_i})}{f_{\lambda_i}} =
  \frac{r_{\lambda_i}}{\lcm(d_{\lambda_i},f_{\lambda_i})},
  \]
  which is clearly integral.
\end{proof}
 Above, we observed that types
$d_{\lambda_i}x \equiv d_{\lambda_i}y \pmod {r_{\lambda_i}}$ iff
$x \equiv y \pmod{r_{\lambda_i}/d_{\lambda_i}}$, so given the lemma,
if we choose orders of types $d_{\lambda_i}x$ with $x$ modulo
$f_{\lambda_i}/{\gcd(d_{\lambda_i}, f_{\lambda_i})}$, they will be
distinct modulo both $r_{\lambda_i}$ and $f_{\lambda_i}$.

We want to fix maximal orders $\R_{\lambda_i}^m$ where
$m \in \Z/f_{\lambda_i}\Z$; we separate those residues which can be
written as $m \equiv d_{\lambda_i}a \pmod{f_{\lambda_i}}$ from those
that cannot.  We put
$\R_{\lambda_i}^{d_{\lambda_i}a}
\colonequals \End_{\Delta_{\lambda_i}}([\L^a])$
for
$a = 0, 1, \dots, f_{\lambda_i}/{\gcd(d_{\lambda_i}, f_{\lambda_i})}
-1$,
and for $m$ one of the remaining
$f_{\lambda_i} - f_{\lambda_i}/{\gcd(d_{\lambda_i}, f_{\lambda_i})}$
residues, choose a maximal order associated to a homothety class of
lattice having type $m$.  Recall that
$f_{\lambda_i} \mid r_{\lambda_i}$, so these choices are possible.

\begin{remark}\label{rem:embedsinab} 
  We note from our remarks above, that $\O_L$ is a subset of
  $\R_{\mu_j}(1,m)$ for every value of $m$, and of $\R_{\nu_k}(m',1)$
  for $0\le m'\le n$.
\end{remark}

Now we use the local-global correspondence for orders to define global
orders from the above local factors. Fix the following notation:

\begin{align*}
  \ba &= (a_i) \in \Z/|\rho_1 H|\Z \times \cdots \times \Z/|\rho_r
        H|\Z,\\
  \bb &= (b_j) \in \Z/|\sigma_1 \wH|\Z \times \cdots \times \Z/
        |\sigma_s \wH|\Z,\\
  \bc &= (c_k) \in \Z/|\tau_1|\Z \times \cdots \times \Z/|\tau_t|\Z.
\end{align*}
Here we assume the coordinates $a_i, b_j, c_k$ are integers which are
the least non-negative residues corresponding to the moduli.  Define
maximal orders, $\D^{\ba,\bb,\bc}$, in $B$ via the local-global
correspondence:
\begin{equation}\label{eq:parametrizedorders}
  \D_\fp^{\ba,\bb,\bc} =
  \begin{cases}
    \R_\fp& \textrm{if }\fp \notin \{\lambda_i, \mu_j, \nu_k\},\\
    \R_{\lambda_i}^{a_i}& \textrm{if } \fp = \lambda_i, i=1, \dots, r,\\
    \R_{\mu_j}^{b_j} \colonequals \R_{\mu_j}(1,b_j)& \textrm{if } \fp = \mu_j, j=1, \dots, s,\\
    \R_{\nu_k}^{c_k} \colonequals \R_{\nu_k}(c_k,1)& \textrm{if } \fp = \nu_k, k=1, \dots, t.\\
  \end{cases}
\end{equation}

We claim that such a collection of maximal orders parametrizes the
isomorphism classes of maximal orders in $B$.  That is, given any
maximal order $\E$ in $B$, we show there are unique tuples
$\ba, \bb, \bc$ so that $\E \cong \D^{\ba,\bb,\bc}$.  To see this we
again employ a local-global principle.  We know that any two maximal
orders in $B$ are equal at almost all places of $K$, so they are
distinguished at only a finite number of places.  We collect
information about those differences by defining a ``distance idele''
associated to the two orders.

Let $\fM$ denote the set of all maximal orders in $B$, and let
$\R_1, \R_2 \in \fM$.  For each place $\nu$ of $K$ we want to define a
local ``type distance'', $td_\nu(\R_{1\nu}, \R_{2\nu})$, which
distinguishes the local orders.  For infinite places $\nu$,
$\R_{1\nu} = \R_{2\nu} = B_\nu$, so (whatever the definition at other
places) it makes sense to define $td_\nu(\R_{1\nu}, \R_{2\nu}) = 0$ in
this case.  We adopt the same convention for a finite place which
totally ramifies in $B$, since there is a unique maximal order in
$B_\nu$.  In the cases where a finite place splits or partially
ramifies, we have already defined the type distance
$td_\nu(\R_{1\nu}, \R_{2\nu})$ in section~\ref{sec:localtheory}.  In
particular, $td_\nu$ is only well-defined modulo $r_\nu$, but this
causes no difficulty.

To return to the problem of parametrizing the isomorphism classes of
maximal orders in $B$, we define a map (called the $G_\R$-valued
distance idele) $\delta: \fM \times \fM \to G_\R = J_K/H_\R$ (where
$H_\R = K^\times nr(\fN(\R))$) as follows: Given $\R_1$,
$\R_2\in \fM$, let $\delta(\R_1, \R_2)$ be the image in $G_\R$ of the
idele $(\pi_\nu^{td_\nu(\R_{1\nu}, \R_{2\nu})})_\nu$, where $\pi_\nu$
is a fixed uniformizing parameter in $K_\nu$ (putting $\pi_\nu=1$ at
the archimedean places).  Note that while the idele is not
well-defined, its image in $G_\R$ is, since at any place where the
type distance might be nontrivial, the local factor in $H_\R$ equals
$\O_\nu^\times (K_\nu^\times)^{r_\nu}$.

That the orders $\{\D^{\ba, \bb, \bc}\}$ parametrize the isomorphism
classes of maximal orders in $B$ follows from the the following
proposition.

\begin{prop}\label{prop:parametrize}
  Let $\R$ be a fixed maximal order in $B$, and consider the
  collection of maximal orders $\D^{\ba,\bb,\bc}$ defined above.

  \begin{compactenum}
  \item If $\E$ is a maximal order in $B$ and $\E \cong \R$, then
    $\delta(\R, \E)$ is trivial.
  \item If $\E \cong \E'$ are maximal orders in $B$, then
    $\delta(\R, \E) = \delta(\R,\E')$.
  \item $\D^{\ba,\bb,\bc} \cong \D^{\ba',\bb',\bc'}$ if and only if
    $\ba = \ba'$, $\bb = \bb'$, and $\bc = \bc'$.
  \item If $\E$ and $\E'$ are maximal orders in $B$, and
    $\delta(\R, \E) = \delta(\R,\E')$, then $\E \cong \E'$.
  \end{compactenum}
\end{prop}

\begin{proof}For the first assertion, we may assume that
  $\E = b\R b^{-1}$ for some $b \in B^\times$ by Skolem-Noether.  Thus
  for each place $\nu$, $\E_\nu = b \R_\nu b^{-1}$.  The goal is to
  show that $\delta(\R, \E) = 1$ in $G_\R$, by showing that the
  distance idele which is derived from the local type distances is the
  same as the principal idele $(nr_{B/K}(b))$ which lies in the image
  of $K^\times$ in $J_K$.

  We first verify that the local factors of the principal idele
  $(nr_{B/K}(b))$ are also trivial in $G_\R$.  Indeed the local
  factors in $G_\R$ are trivial at both the infinite and totally
  ramified places with the possible exception of a real place which
  ramifies in $B$, but it follows from (33.4) of \cite{Reiner-book}
  that the norm is positive which is trivial in the local factor
  $\mathbb R^\times/\mathbb R_+^\times$.

  Thus we need only consider places $\nu$ which are split or partially
  ramified in $B$.  We handle these cases together as in our
  description above, and assume $B_\nu$ has been identified with
  $\Mat_{r_\nu}(D_\nu)$ where $D_\nu$ is a central division algebra of
  degree $m_\nu$ over $K_\nu$. As before we let $\Delta_\nu$ be the
  unique maximal order in $D_\nu$.  For convenience assume that the
  identification of $B_\nu$ with $\Mat_{r_\nu}(D_\nu)$ is done in such
  a way that, as described in the previous section, there is a rank
  $r_\nu$ free $\Delta_\nu$-lattice $\Lambda_\nu$ so that
  $\R_\nu = \End_{\Delta_\nu}(\Lambda_\nu)$, and hence
  $\E_\nu = \End_{\Delta_\nu}(b\Lambda_\nu)$ for some
  $b \in \GL_{r_\nu}(D_\nu)$.  Using elementary divisors
  for $\Delta_\nu$-lattices, we may assume without loss that
  $b = \diag(\bpi_{D_\nu}^{a_1}, \dots, \bpi_{D_{\nu}}^{a_{r_\nu}})$.
  Then
  $td_\nu(\R_\nu,\E_\nu) \equiv \sum_{i=1}^{r_\nu} a_i \pmod {r_\nu}.$

  So we shall compare the cosets
  $\pi_\nu^{\sum_{i=1}^{r_\nu} a_i} \O_\nu(K_\nu^\times)^{r_\nu}$ with
  $\pi_\nu^\ell \O_\nu(K_\nu^\times)^{r_\nu}$ where
  $\ell = \ord_{\pi_\nu}(nr_{B_\nu/K_\nu}(b))$.

  We check that indeed
  $\ell \equiv \sum_{i=1}^{r_\nu} a_i\pmod {r_\nu}$ as follows.  With
  $b = \diag(\bpi_{D_\nu}^{a_1}, \dots, \bpi_{D_{\nu}}^{a_{r_\nu}})
  \in B_\nu^\times=GL_{r_\nu}(D_\nu)$,
  we recall from earlier
  $nr_{D_\nu/K_\nu}(\bpi_{D_\nu}) = (-1)^{m_\nu-1}\pi_\nu$, so (up to
  units in $\O_\nu$) $nr_{B_\nu/K_\nu}(b) =\pi_\nu^{\sum a_i}$, hence
  the result.

   Thus we see that $\delta(\R, \E)$ is the image in
  $G_\R$ of the principal idele $(nr_{B/K}(b))_\nu$, so
  $\delta(\R,\E) = 1$ in $G_\R = J_k/K^\times nr(\fN(\R))$ as
  $(nr_{B/K}(b))_\nu$ is in the image of $K^\times$ in $J_K$.

  For the second claim, we may write $\E' = b\E b^{-1}$ for some
  $b \in B^\times$, so $\E'_\nu = b\E_\nu b^{-1}$ for each place
  $\nu$, and as in the previous part, we need only worry about those
  places $\nu$ which split or are partially ramified in $B$.  So as
  before, we write $\R_\nu = \End_{\Delta_\nu}(\Lambda_\nu)$ and
  $\E_\nu = \End_{\Delta_\nu}(\Gamma_\nu)$, so that
  $\E'_\nu = \End_{\Delta_\nu}(b\Gamma_\nu)$, where $\Lambda_\nu$ and
  $\Gamma_\nu$ are free $\Delta_\nu$-lattices of rank $r_\nu$.
  Considering the invariant factors of the lattices $\Lambda_\nu$,
  $\Gamma_\nu$ and $b\Gamma_\nu$, we easily see that
  \[\delta(\R, \E') = \delta(\R,\E)\delta(\E,\E') = \delta(\R, \E),\]
  since $\delta(\E, \E') = 1$ by the computations in the first part.

  For the third statement, we need only show one direction.  Let $\nu$
  be a finite place of $K$ and $\pi_\nu$ the corresponding
  uniformizing parameter of $K_\nu$. Let $\tilde\omega_\nu$ denote the
  idele with $\pi_\nu$ in the $\nu$th place and 1's elsewhere.
  Observe that Artin reciprocity identifies the image of
  $\tilde\omega_\nu$ in $G_R = J_K/H_\R$ with the Artin symbol
  $(\nu, K(\R)/K) \in \Gal(K(\R)/K)$.  Moreover, for two maximal
  orders $\E, \E'$ of $B$, we see that $\delta(\E, \E')$ is equal to
  the image of $\prod_\nu \pi_\nu^{td_\nu(\E, \E')}$ in $G_\R$, and
  hence corresponds to a product of Artin symbols.

  We recall that the orders $\D^{\ba,\bb,\bc}$ differ from our fixed
  maximal order $\R$ only at finite places which were unramified in
  both $L$ and $B$.  At such a place $\nu$, we identified $B_\nu$ with
  $\Mat_{r_\nu}(D_\nu)$ and our representative maximal orders were
  identified as endomorphism rings of homothety classes of lattices
  relative to some fixed basis $\{\alpha_i\}$ of $D_\nu^{r_\nu}$. Now
  referring to the conventions we adopted for the places
  $\lambda_i, \mu_j, \nu_k$ whose associated Artin symbols were used
  to parametrize $\Gal(K(\R)/K)$, we check that $\pmod n$,

  \[td_\nu(\delta(\D^{\ba,\bb,\bc}, \D^{\ba',\bb',\bc'})\equiv
  \begin{cases}
    a_i' - a_i&\mbox{for } \nu = \lambda_i,\\
    b_j' - b_j&\mbox{for } \nu = \mu_j,\\
    c_k'-c_k&\mbox{for } \nu = \nu_k.
  \end{cases}
  \]

  It follows that
  \[
  \delta(\D^{\ba,\bb,\bc}, \D^{\ba',\bb',\bc'}) \leftrightarrow
  \prod_{i=1}^g \rho_i^{a_i' - a_i} \prod_{j=1}^s \sigma_j^{b_j'-b_j}
  \prod_{k=1}^t \tau_k^{c_k'-c_k} \in \Gal(K(\R)/K),
  \]
  which is trivial if and only if $\ba=\ba'$, $\bb = \bb'$, and
  $\bc=\bc'$ by Proposition~\ref{prop:Galoisrepn}.

  Finally for the last statement, let $\E$ and $\E'$ be maximal orders
  in $B$ with $\delta(\R, \E) = \delta(\R,\E')$.  Suppose to the
  contrary that $\E \not\cong \E'$.  Then $\E \cong \D^{\ba,\bb,\bc}$,
  $\E' \cong \D^{\ba',\bb',\bc'}$ where at least one of $\ba,\bb,\bc$
  differs from $\ba', \bb', \bc'$. Since $\R = \D^{\0,\0,\0}$, the
  computations above show that
  $\delta(\R, \D^{\ba,\bb,\bc}) \ne \delta(\R, \D^{\ba',\bb',\bc'})$,
  but by part (2) of the proposition
  $\delta(\R,\E) = \delta(\R, \D^{\ba,\bb,\bc})$, and
  $\delta(\R,\E') = \delta(\R, \D^{\ba',\bb',\bc'})$, which provides
  the desired contradiction.  This completes the proof.
\end{proof}

 We now summarize our efforts in this section labeling
those isomorphism classes of maximal orders in $B$ which contain (a
representative containing) the ring of integers $\O_L$.  Above we have
parametrized the isomorphism classes of maximal orders by the set
$\{\D^{\ba,\bb,\bc}\}$ given in
equation~(\ref{eq:parametrizedorders}).  These orders are locally
equal to $\R$ at all places except those designated previously as a
member of the set
$T=\{\lambda_1, \dots, \lambda_r, \mu_1, \dots,\mu_s, \nu_1,
\dots\nu_t\}$.
By this assumption, for $\fp \notin T$, we have
$\O_L \subset \D_\fp^{\ba,\bb,\bc}$.  For $\fp = \mu_j$ or $\nu_k$, we
also have $\O_L \subset \D_\fp^{\ba,\bb,\bc}$ by
Remark~\ref{rem:embedsinab}. Finally,
$\O_L \subset \R_{\lambda_i} = \D_{\lambda_i}^{\mathbf 0,\bb,\bc}$ for
all the places $\lambda_i$.  Thus, for all finite $\fp$ in $K$,
$\O_L \subset \D_{\fp}^{\ba,\bb,\bc}$ for all $\bb, \bc$, and
$\ba = \mathbf 0$, and so by the local-global correspondence,
$\O_L \subset \D^{\mathbf 0,\bb,\bc}$ for all $\bb, \bc$.  But these
orders $\{\D^{\mathbf 0,\bb,\bc}\}$ are precisely those which
correspond to the elements of $H = \Gal(K(\R)/L_0)$.  We summarize
this as

\begin{thm}\label{thm:lowerbound} The ring of integers, $\O_L$ is
  contained in at least $[K(\R):L_0]$ of the $[K(\R):K]$
  representatives $\{\D^{\ba,\bb,\bc}\}$. Specifically,
  $\O_L \subset \D^{\mathbf 0,\bb,\bc}$ for all $\bb, \bc$.
\end{thm}

\section{ Recovering global selectivity results}

In this section we recover and refine some global results on selective
orders.  Recall that we have a central simple algebra $B = M_r(D)$
where $D$ is a central division algebra of degree $m$ over a number
field $K$. We have an extension $L/K$ of degree $n =rm$ which embeds
in $B$, and we have fixed a maximal order $\R$ of $B$ which contains
$\O_L$.  Associated to $\R$ is a class field, $K(\R)$, and we have set
$L_0 = K(\R) \cap L$.  We assume $n \ge 3$.

\subsection{Simple lower bounds}
It is immediate from Theorem~\ref{thm:lowerbound}, that the ring of
integers, $\O_L$ is contained in at least $[K(\R):L_0]$ maximal orders
which lie in distinct isomorphism classes, so speaking informally, at
least $1/[L_0:K]$ of the isomorphism classes ``admit an embedding'' of
$\O_L$.  

Having established $1/[L_0:K]$ as a lower bound, we next show that the
degree $[L_0:K]$ is further constrained as a divisor of
$[L:K] = n = rm$ in an interesting way.

\begin{prop}Let $B = M_r(D)$ where $D$ is a central division algebra
  of degree $m$ over a number field $K$ which contains an extension
  $L/K$ of degree $n = rm$. Fix a maximal order $\R$ of $B$ which
  contains the ring of integers $\O_L$.  As above, associate to $\R$ a
  class field extension $K(\R)/K$, and put $L_0 = L \cap K(\R)$.  If
  no real place of $K$ ramifies in $B$, then
  $[L_0:K] \mid r\cdot\gcd(r,m)$; otherwise
  $[L_0:K] \mid 2r\cdot\gcd(r,m)$.  In particular if $\gcd(r,m) = 1$,
  then $[L_0:K] \mid r$ or $2r$.
\end{prop}
\begin{remark}
  The proposition above extends the simplest form of Carmona's
  \cite{Arenas-Carmona-2014-Selectivity-Division-Algebras} result,
  where he shows that for an arbitrary division algebra, the
  selectivity proportion is 1/2 or 1, which we see from above with
  $r = 1$.
\end{remark}

\begin{proof}
  This proof follows the lines of a similar argument in
  \cite{Arenas-Carmona-2014-Selectivity-Division-Algebras}.  For each
  place $\nu$ of $K$, we have written $B_\nu \cong M_{r_\nu}(D_\nu)$
  where $D_\nu$ is a central division algebra of degree $m_\nu$ over
  $K_\nu$, and of course where $n = rm = r_\nu m_\nu$.  By (32.17) of
  \cite{Reiner-book}, we know that $m = \lcm\{m_\nu\}$ where the lcm
  is taken over all places of $K$.

  To begin, let $p$ be an odd prime, and assume $p^t \| m$, $t \ge 1$.
  Also assume that $p^s \| r$ with $s \ge 0$.  Then there must be a
  place $\nu$ of $K$ with $p^t \| m_\nu$.  Since $p$ is odd, we know
  $\nu$ is a finite place of $K$.  Since $L$ embeds in $B$, we know
  for every place $\fP$ lying above $\nu$ that
  \[m_\nu \mid [L_\fP:K_\nu] = [L_\fP:(L_0)_{\fP\cap
    L_0}][(L_0)_{\fP\cap L_0}: K_\nu] = [L_\fP:(L_0)_{\fP\cap L_0}]
  f(\nu; L_0/K),\]
  the last equality since $K(\R)/K$ is abelian and unramified at all
  finite places.  By
  Remark~\ref{remark:classfieldunramifield-orderArtinsymbol}, we know
  that $f(\nu; K(\R)/K) \mid r_\nu$, hence so does $f(\nu; L_0/K)$.
  Now since $n = rm = r_\nu m_\nu$ and $p^t \| m_\nu$ we have
  $p^s \| r_\nu$, so $\ord_p(f(\nu;L_0/K)) \le s$.  Let
  $t_0 = \max\{0, t-s\}$.  Then \newline
  $p^{t_0} \mid [L_\fP: (L_0)_{\fP\cap L_0}]$.  It follows that
  $p^{t_0} \mid [L:L_0]$. Therefore
  \[
  \ord_p[L_0:K] \le s+t - t_0 =
  \begin{cases}
    2s& s\le t\\
    s+t& s > t\\
  \end{cases}
  \ =\ord_p(r) + \ord_p(\gcd(r,m)).
  \]
  Which gives the result for the odd primes $p$.  When $p = 2$, if
  $4 \mid m$, the same argument gives the correct bounds with
  $p=2$. Moreover, even if $2\| m$, but there is some finite place
  $\nu$ with $2\mid m_\nu$, the argument is valid.  It is only in the
  case that $2 \| m$, but for no finite place does $2 \mid m_\nu$ that
  the argument fails, and in that case we must have a real place which
  ramifies in~$B$.
\end{proof}

\subsection{The effect of ramification on the bounds}
The ramification of the central simple algebra $B$ has an interesting
impact on selectivity.  In
Theorem~\ref{thm:divisionalgebra-noselectivity}, we show that if there
is a finite place of $K$ which is totally ramified in $B$, there is
never selectivity; that is, every isomorphism class of maximal orders
in $B$ admits an embedding of $\O_L$.  At the other end of the
spectrum, if for each finite place of $K$, $B$ is split, then the
selectivity proportion is either 1 (no selectivity) or $1/[L_0:K]$.
In the case of a central simple algebra $B$ which has partial
ramification at some places, the proportion of isomorphism classes
which admit an embedding of $\O_L$ will be of the form $m/[L_0:K]$ for
an integer $m$ which is the cardinality of a certain subgroup of
$\Gal(L_0/K)$ related to the finite places of $K$ which are partially
ramified in $B$.

Let's begin with the case of a totally ramified prime.  
This theorem was proven for algebras of odd prime degree in
\cite{Linowitz-Shemanske-EmbeddingPrimeDegree}, but remains valid for
general degree $n\ge 3$.

\begin{thm}\label{thm:divisionalgebra-noselectivity}
  Suppose there is a finite place $\nu$ of $K$ which is totally
  ramified in $B$, that is, $m_\nu=n$. Let $\Omega \subset \O_L$ be
  any $\O_K$-order.  Then every maximal order in $B$ admits an
  embedding of $\Omega$.  In particular, there can never be
  selectivity.
\end{thm}

\begin{proof}
  It is enough to show that every maximal order in $B$ admits an
  embedding of $\O_L$.  Since $B_\nu$ is a division algebra, there is
  a unique maximal order $\R_\nu$ in $B_\nu$ whose normalizer is all
  of $B_\nu^\times$ and so $K_\nu^\times$, the norm of the normalizer,
  is contained in $H_\R$.  This means that that $\nu$ splits
  completely in the class field $K(\R)$, hence also in
  $L_0 = K(\R)\cap L$.

  On the other hand, by the Albert-Brauer-Hasse-Noether theorem,
  $m_\nu=n \mid [L_\fP:K_\nu]$ for all places $\fP$ of $L$ lying above
  $\nu$.  This means that $\nu$ is inert in $L$, hence also in $L_0$.
  Since $L_0/K$ is unramified (at $\nu$), we have
  $[L_0:K] = f(\fP|\nu)$.  But $\nu$ splits completely in $L_0$, so
  $[L_0:K] = f(\fP|\nu) =1$, and the result is now immediate from
  Theorem~\ref{thm:lowerbound}.
\end{proof}

To go further, we shall utilize the notion of the distance idele and
Proposition~\ref{prop:parametrize} to characterize those isomorphism
classes of maximal orders which admit an embedding of $\O_L$.  We have
assumed that $\O_L \subset \R$.  If there is an embedding of $\O_L$
into a maximal order $\E$, then $\O_L$ is contained in a conjugate
maximal order, $\E'$, and by Proposition~\ref{prop:parametrize}, the
distance ideles $\delta(\R, \E)$ and $\delta(\R,\E')$ are equal.  So
the idea is to assume that $\O_L$ is contained in maximal orders $\R$
and $\E$, and to consider their distance idele
$\delta(\R, \E) \in G_\R$.  Recall that $G_\R \cong \Gal(K(\R)/K)$,
and that we parametrized the isomorphism classes of maximal orders in
$B$ with representatives $\D^{\ba,\bb,\bc}$ having the property that
viewing the distance idele as an element of $\Gal(K(\R)/K)$ we have
(see Proposition~\ref{prop:Galoisrepn})
\[\delta(\R, \D^{\ba,\bb,\bc}) =
\rho_1^{a_1} \cdots \rho_r^{a_r}\sigma_1^{b_1} \cdots \sigma_s^{b_s}
\tau_1^{c_1}\cdots \tau_t^{c_t}.
\]

In Theorem~\ref{thm:lowerbound}, we see that $\O_L$ is always
contained in those representatives where
\[\delta(\R, \D^{\ba,\bb,\bc}) =
\rho_1^{0} \cdots \rho_r^{0}\sigma_1^{b_1} \cdots \sigma_s^{b_s}
\tau_1^{c_1}\cdots \tau_t^{c_t},
\]
that is, those elements whose distance idele lies in
$H = \Gal(K(\R)/L_0) \le G=\Gal(K(\R)/K)$.  To delve more deeply, we
now view $\delta(\R,\E)|_{L_0} \in \Gal(L_0/K) \cong G/H$.  We sketch
the framework we employ.

Recall some notation from the introduction.  Given a finite place
$\nu$ of $K$, and the local index $m_\nu$, we know that
$m_\nu \mid [L_\fP:K_\nu]$ for all places $\fP$ of $L$ lying above
$\nu$.  Further, we set $r_\fP = [L_\fP:K_\nu]/m_\nu$. Next, we
defined:
\begin{align}
  d_{\nu} = \gcd_{\fP\mid \nu}r_\fP &= 
                                      \gcd_{\fP\mid \nu} \frac{[L_\fP:(L_0)_{\fP\cap L_0}][
                                      (L_0)_{\fP\cap L_0}:K_{\nu}]}{m_{\nu}}\\
                                    &= \gcd_{\fP\mid \nu} \frac{[L_\fP:(L_0)_{\fP\cap L_0}]
                                      f(\nu; L_0/K)}{m_{\nu}}\\
                                    &= \gcd_{\fP\mid \nu}([L_\fP:(L_0)_{\fP\cap
                                      L_0}])\frac{f(\nu; L_0/K)}{m_{\nu}}.
                                      \label{eq:Thequantityd} 
\end{align}

Now recall that the type distance, $\delta(\R, \E)$, is the image of
the idele $(\pi_\nu^{td_\nu(\R_\nu,\E_\nu)})_\nu$ in $G_\R$, and
viewed as an element of $\Gal(K(\R)/K)$ it is a product of (powers of)
Artin symbols.  So we can view
\[\delta(\R, \E)|_{L_0} = \prod_{\nu \mbox{ \scriptsize finite}} (\nu,
L_0/K)^{td_\nu(\R_\nu,\E_\nu)},
\]
where we recall that the Artin symbol, $(\nu, L_0/K)$, has order equal
to the inertia degree $f(\nu; L_0/K)$.  Finally, from
Theorem~\ref{thm:OLembedslocally}, we know that if
$\O_L \subset \R\cap \E$, and $\nu$ is unramified in $L$, then
$td_\nu(\R_\nu, \E_\nu)$ will be divisible by $d_\nu$.  Now consider
Equation (\ref{eq:Thequantityd}).  If $m_\nu = 1$ (that is, if
$B_\nu \cong M_n(K_\nu)$), then $d_\nu$ is divisible by
$f(\nu; L_0/K)$, the order of $(\nu, L_0/K)$, so that factor in
$\delta(\R,\E)|_{L_0}$ will be trivial.  So we see it is here that the
partially ramified primes play a critical role in producing a
selectivity proportion strictly between $1/[L_0:K]$ and 1.

Motivated by the above remarks, let $\lambda_1, \dots, \lambda_\ell$
be the set places which are partially ramified in $B$.
\begin{remark}
  In order to use Theorem~\ref{thm:OLembedslocally} below, we must
  also assume that the $\lambda_i$ are all unramified in $L$.
\end{remark}

For each place, $\lambda_i$, we have the quantity $d_{\lambda_i}$ from
Equation~(\ref{eq:Thequantityd}).  Let $G_0$ be the subgroup of
$\Gal(L_0/K)$ generated by the Artin symbols:
\[G_0 = \la (\lambda_1, L_0/K)^{d_{\lambda_1}}, \dots, (\lambda_\ell,
L_0/K)^{d_{\lambda_\ell}}\ra \le \Gal(L_0/K).
\]

Write $f_{\lambda_i}$ for $f(\lambda_i; L_0/K)$.  From
equation~(\ref{eq:Thequantityd}), we know that
\[d_{\lambda_i} = \gcd_{\fP\mid \lambda_i}([L_\fP:(L_0)_{\fP\cap
  L_0}])\frac{f_{\lambda_i}}{m_{\lambda_i}},
\]
and we know the order of $(\lambda_i, L_0/K)$ is $f_{\lambda_i}$. So
if
$m_{\lambda_i} \mid \gcd_{\fP\mid \lambda_i}([L_\lambda:(L_0)_{\fP\cap
  L_0}])$,
we know that $(\lambda_i, L_0/K)^{d_{\lambda_i}} = 1\in G_0$;
otherwise it generates a cyclic subgroup of order
$f_{\lambda_i}/\gcd(d_{\lambda_i}, f_{\lambda_i})$.  For our use
below, we want to define maximal orders, $\Gamma_{\lambda_i}^a$, in
the local algebra $B_{\lambda_i}$ with type distance,
$td_{\lambda_i}(R_{\lambda_i},\Gamma_{\lambda_i}^a) = d_{\lambda_i}a$
with
$a = 0, 1, \dots, f_{\lambda_i}/\gcd(d_{\lambda_i}, f_{\lambda_i})-1$.
We do this in exactly the same way as we did in the previous section
just prior to Remark~\ref{rem:embedsinab} where we defined the orders
$\R_{\lambda_i}^{a_i}$, so we do not repeat the argument here,
although we do reiterate that we are assuming that the places
$\lambda_i$ are unramified in $L$ so as to leverage
Theorem~\ref{thm:OLembedslocally}.

\begin{thm}\label{thm:MainSummaryA}\
  Assume that $\O_L \subset \R \subset B$. For every $\sigma \in G_0$,
  there exists a maximal order $\E$ in $B$ so that $\O_L \subset \E$,
  and viewing the distance idele, $\delta(\R,\E)$, as an element of
  $\Gal(K(\R)/K)$, we have that
  $\delta(\R,\E)|_{L_0} =\sigma \in G_0$.
\end{thm}

\begin{proof} Let
  $\sigma_i = (\lambda_i,L_0/K)^{d_{\lambda_i}} \in G_0$ be a
  generator of $G_0$, and write
  $\sigma = \prod_{i=1}^\ell \sigma_i^{a_i}$, where we understand the
  expression may not be unique. Define a maximal order $\E$ of $B$ via
  the local-global correspondence by specifying:
  \[\E_\nu =
  \begin{cases}
    \R_\nu& \mbox {for }\nu \notin \{\lambda_1, \dots, \lambda_\ell\},\\
    \Gamma_{\lambda_i}^{a_i}& \mbox {for } \nu = \lambda_i,\quad i =
    1, \dots, \ell.
  \end{cases}
  \]
  Then, viewing $\delta(\R,\E)$ as an element of $\Gal(K(\R)/K)$, we
  have
  $\delta(\R,\E) = \prod_{i=1}^\ell (\lambda_i;
  K(R)/K)^{d_{\lambda_i}a_i}$,
  so that $\delta(\R,\E)|_{L_0} =\sigma \in G_0$.
\end{proof}

\begin{remark}
  Presuming that $\sigma \ne 1$ in the above theorem,
  $\E \cong \D^{\ba,\bb,\bc}$ for some $\ba \ne \mathbf 0$, meaning
  that the proportion of isomorphism classes admiting an embedding of
  $\O_L$ is greater than $1/[L_0:K]$.  Indeed, this theorem says that
  the proportion is at least $|G_0|/[L_0:K]$.
\end{remark}

Now we would like some sort of converse, meaning if there is
selectivity, then this is an upper bound as well. We have the
following qualified result.

\begin{thm}\label{thm:MainSummaryB}\
  Assume that $\O_L \subset \R \subset B$.  Let $\E$ be another
  maximal order in $B$, and let $\delta(\R,\E)$ denote the distance
  idele. Assume further, that any place $\nu$ for which
  $td_\nu(\R_\nu,\E_\nu) \not\equiv 0 \pmod {r_\nu}$ is unramified in
  $L$.  If $\O_L \subset \E$, then $\delta(\R,\E)|_{L_0} \in G_0$.
\end{thm}

\begin{proof}Let $\delta(\R,\E)\in G_\R=J_K/H_\R$ be the distance
  idele.  Let $\nu$ be any place for which
  $td_\nu(\R_\nu,\E_\nu) \not\equiv 0 \pmod {r_\nu}$.  By assumption,
  we have that $\nu$ is unramified in $L$, and so, by conventions on
  the type distance, $\nu$ is a finite place and not totally ramified
  in $B$.  Since $O_L \subset \E_\nu$, by
  Theorem~\ref{thm:OLembedslocally}, we have that
  $td_\nu(\R_\nu,\E_\nu)$ is divisible by $d_\nu$, which means the
  local factor of the Artin symbol coming from $\delta(\R,\E)$ has the
  form $(\nu; K(\R)/K)^{d_\nu \ell}$ for some integer $\ell$. So
  restricted to $L_0/K$, the Artin symbol becomes
  $(\nu; L_0/K)^{d_\nu\ell}$.  By Equation~(\ref{eq:Thequantityd}), if
  $\nu$ is unramified in $B$, then $m_\nu = 1$ which implies
  $d_\nu \equiv 0 \pmod {f(\nu; L_0/K)}$, but $f(\nu,L_0/K)$ is the
  order of the Artin symbol $(\nu;L_0/K)$, so this factor is trivial.
  The only factors left are those which correspond to partially
  ramified places in $B$, and so it is clear that
  $\delta(\R,\E)|_{L_0} \in G_0$.
\end{proof}

We can summarize the previous two theorems as:
\begin{thm}\label{thm:MainSummary}\
  Let $\lambda_1, \dots, \lambda_\ell$ be the set of finite places of
  $K$ which are partially ramified in $B$. Assume the $\lambda_i$ are
  all unramified in $L$.  Let
  \[G_0 = \la (\lambda_1, L_0/K)^{d_{\lambda_1}}, \dots,
  (\lambda_\ell, L_0/K)^{d_{\lambda_\ell}}\ra \le \Gal(L_0/K),
  \]
  be the subgroup generated by powers of the Artin symbols
  $(\lambda_i, L_0/K)$.  The proportion of isomorphism classes of
  maximal orders which admit an embedding of $\O_L$ is at least
  $\ds\frac{|G_0|}{[L_0:K]}$, and if $L \subseteq K(\R)$ (so in
  particular, $L$ is unramified at all the finite places of $K$), then
  the proportion is exactly $\ds\frac{|G_0|}{[L_0:K]}$.
\end{thm}

\section{An Example}

We give a simple example of
Theorem~\ref{thm:MainSummary}. Computations are done with Magma
\cite{Magma}.

Let $K = \Q(\sqrt{-39})$.  Then the ideal class group of $K$ is cyclic
of order 4, hence the Hilbert class field of $K$, $H_K$ has Galois
group, $\Gal(H_K/K)$, cyclic of order 4. The rational prime 61 splits
completely in $K$, and there are four primes of $H_K$ lying above 61.
So put $61\O_K = \fp_1\fp_2$.  Since $H_K/K$ is Galois, the only way
for $61\O_{H_K}$ to factor as the product of four distinct primes in
$H_K$ is for each of the primes $\fp_i$ to have inertia degrees
$f(\fp_1; H_K/K) = f(\fp_2; H_K/K) =2$.

To construct our central simple algebra, we specify Hasse invariants.
Let $m_{\fp_1} = m_{\fp_2} = 2$ and $m_\nu = 1$ for all other places
$\nu$ of $K$.  Taking Hasse invariants $1/m_\nu$ for all places $\nu$
of $K$, the short exact sequence of Brauer groups (e.g., (32.13) of
\cite{Reiner-book}) guarantees the existence of a degree 4 central
simple $K$-algebra $B = M_2(D)$ having the prescribed Hasse invariants.

Let $L = H_K$. The field $L$ satisfies the conditions of the
Albert-Brauer-Hasse-Noether theorem, so $L$ embeds in $B$ as a
$K$-algebra.  Now let $\R$ be any maximal order of $B$ which contains
$\O_L$, and $K(\R)$ the associated class field.

Since $K$ has no real embeddings, its narrow class field and its
Hilbert class field coincide, so $K(\R) \subseteq H_K$.

To show the reverse containment, recall that the class field $K(\R)$
arises field class field theory via the quotient $J_K/H_\R$ where
$H_\R$ is characterized by information about the local norm of
normalizers of the $\R_\nu$ which we characterized in section 3.  It
is then easy to check that the class group associated to $H_K$
contains $H_\R$, so $H_K \subseteq K(\R)$.

Thus $L = H_K = K(\R) = L_0$.

We now refer to the notation of Theorem~\ref{thm:MainSummary}.  We
have $\lambda_1 = \fp_1$ and $\lambda_2 = \fp_2$ and via
Equation~(\ref{eq:Thequantityd}), compute
$d_{\lambda_1} = d_{\lambda_2} = 1$. So $G_0$ is generated by the
Artin symbols $(\fp_1, H_K/K)$ and $(\fp_2, H_K/K)$ each of which has
order 2, but as $\Gal(H_K/K)$ is cyclic of order 4, they must be
equal, so that $|G_0| = 2$.  So while the standard lower bound for the
selectivity proportion is $1/[L_0:K] = 1/4$, we have
$|G_0|/[L_0:K] = 2/4 = 1/2$.

% Just the fact that we (now) know the local normalizers. For instance
% H_K /K is everywhere unramified so the associated idelic group
% contains units at each coordinate. Similarly the extension is of
% exponent 4 so contains K_v^4 everywhere. It has the Frobenius
% elements of P_1 and P_2 with inertia degree 2, so it has to contain
% squares locally at P_1 and P_2. So it has to contain the idelic
% group associated to K(R). And since the bijection between subgroups
% of J_K and class fields is inclusion reversing we get the claim.

% I don't think we have to say the above explicitly since I believe it
% is standard (idelic) class field theory. But that's what I had in
% mind.

% ------------------------------------------- Begin References
% -------------------------------------------

\bibliographystyle{amsplain}
% \bibliography{shemanske}

\begin{thebibliography}{10}
\bibitem{Abramenko-Nebe} Peter Abramenko and Gabriele Nebe,
  \emph{Lattice chain models for affine buildings of classical type},
  Math. Ann. \textbf{322} (2002), no.~3, 537--562. \MR{MR1895706
    (2003a:20048)}



\bibitem{Arenas-Carmona-2003} Luis Arenas-Carmona, \emph{Applications
    of spinor class fields: embeddings of orders and quaternionic
    lattices}, Ann. Inst. Fourier (Grenoble) \textbf{53} (2003),
  no.~7, 2021--2038. \MR{MR2044166 (2005b:11044)}

\bibitem{Arenas-Carmona-2011-Representation-Fields} \bysame,
  \emph{Representation fields for commutative orders}, Ann. Inst.
  Fourier (Grenoble) \textbf{62} (2012), no.~2, 807--819. \MR{2985517}


\bibitem{Arenas-Carmona-2014-Selectivity-Division-Algebras} \bysame,
  \emph{Selectivity in division algebras}, Arch. Math. (Basel)
  \textbf{103} (2014), no.~2, 139--146. \MR{3254357}


\bibitem{Artin-Tate} Emil Artin and John Tate, \emph{Class field
    theory}, AMS Chelsea Publishing, Providence, RI, 2009, Reprinted
  with corrections from the 1967 original.  \MR{MR2467155
    (2009k:11001)}

\bibitem{Atiyah-Macdonald} M.~F. Atiyah and I.~G. Macdonald,
  \emph{Introduction to commutative algebra}, Addison-Wesley
  Publishing Co., Reading, Mass.-London-Don Mills, Ont., 1969.

  \MR{0242802 (39 \#4129)}

\bibitem{Ballantine-Rhodes-Shemanske} Cristina~M. Ballantine,
  John~A. Rhodes, and Thomas~R. Shemanske, \emph{Hecke operators for
    {$\GL_n$} and buildings}, Acta Arithmetica \textbf{112} (2004),
  131--140.

\bibitem{Magma} Wieb Bosma, John Cannon, and Catherine Playoust,
  \emph{The {M}agma algebra system. {I}. {T}he user language},
  J. Symbolic Comput. \textbf{24} (1997), no.~3-4, 235--265,
  Computational algebra and number theory (London, 1993).
  \MR{MR1484478}

\bibitem{Brown} Kenneth~S. Brown, \emph{Buildings}, Springer-Verlag,
  New York, 1989.  \MR{MR969123 (90e:20001)}

\bibitem{Chan-Xu} Wai~Kiu Chan and Fei Xu, \emph{On representations of
    spinor genera}, Compos.  Math. \textbf{140} (2004), no.~2,
  287--300. \MR{MR2027190 (2004j:11035)}

\bibitem{Chevalley-book} C.~Chevalley, \emph{Algebraic number fields},
  L'arithm\'etique dan les alg\`ebres de matrices, Herman, Paris,
  1936.


\bibitem{Chinburg-Friedman} Ted Chinburg and Eduardo Friedman,
  \emph{An embedding theorem for quaternion algebras}, J. London
  Math. Soc. (2) \textbf{60} (1999), no.~1, 33--44.  \MR{MR1721813
    (2000j:11173)}

\bibitem{Cohn-Algebrav3} P.~M. Cohn, \emph{Algebra. {V}ol. 3}, second
  ed., John Wiley \& Sons, Ltd., Chichester, 1991. \MR{1098018
    (92c:00001)}



\bibitem{Doyle-Linowitz-Voight} P.~Doyle, B.~Linowitz, and J.~Voight,
  \emph{Minimal isospectral and nonisometric 2-orbifolds}, (preprint).

\bibitem{Frohlich-lf} A.~Fr{\"o}hlich, \emph{Locally free modules over
    arithmetic orders}, J. Reine Angew. Math. \textbf{274/275} (1975),
  112--124, Collection of articles dedicated to Helmut Hasse on his
  seventy-fifth birthday, III. \MR{MR0376619 (51 \#12794)}


  % \bibitem{Gordon-IsospectralSurvey} Carolyn S.\ Gordon, \newblock
  %   \emph{Survey of isospectral manifolds}, \newblock Handbook of
  %   differential geometry, {V}ol.\ {I}, North-Holland, Amsterdam,
  %   2000, 747--778.
        
\bibitem{Guo-Qin} Xuejun Guo and Hourong Qin, \emph{An embedding
    theorem for {E}ichler orders}, J. Number Theory \textbf{107}
  (2004), no.~2, 207--214. \MR{MR2072384 (2005c:11141)}

\bibitem{Higgins} P.~J. Higgins, \emph{Introduction to topological
    groups}, Cambridge University Press, London, 1974, London
  Mathematical Society Lecture Note Series, No. 15.  \MR{MR0360908 (50
    \#13355)}

\bibitem{Lang-ANT} Serge Lang, \emph{Algebraic number theory}, second
  ed., Graduate Texts in Mathematics, vol. 110, Springer-Verlag, New
  York, 1994. \MR{MR1282723 (95f:11085)}

\bibitem{Linowitz-selectivity} B.~Linowitz, \emph{Selectivity in
    quaternion algebras}, J. of Number Theory \textbf{132} (2012),
  1425--1437.

\bibitem{Linowitz-Shemanske-EmbeddingPrimeDegree} Benjamin Linowitz
  and Thomas~R. Shemanske, \emph{Embedding orders into central simple
    algebras}, J. Th\'eor. Nombres Bordeaux \textbf{24} (2012), no.~2,
  405--424. \MR{2950699}

\bibitem{Lubotzky-Samuels-Vishne} Alexander Lubotzky, Beth Samuels,
  and Uzi Vishne, \emph{Division algebras and noncommensurable
    isospectral manifolds}, Duke Math. J. \textbf{135} (2006), no.~2,
  361--379. \MR{2267287 (2008h:11050)}


\bibitem{Maclachlan} C.~Maclachlan, \emph{Optimal embeddings in
    quaternion algebras}, J. Number Theory \textbf{128} (2008),
  2852--2860.

\bibitem{Narkiewicz-book} W{\l}adys{\l}aw Narkiewicz, \emph{Elementary
    and analytic theory of algebraic numbers}, second ed.,
  Springer-Verlag, Berlin, 1990. \MR{MR1055830 (91h:11107)}

\bibitem{Pierce-book} Richard~S. Pierce, \emph{Associative algebras},
  Graduate Texts in Mathematics, vol.~88, Springer-Verlag, New York,
  1982, , Studies in the History of Modern Science, 9. \MR{MR674652
    (84c:16001)}



\bibitem{Reiner-book} I.~Reiner, \emph{Maximal orders}, Academic Press
  [A subsidiary of Harcourt Brace Jovanovich, Publishers], London-New
  York, 1975, London Mathematical Society Monographs,
  No. 5. \MR{MR0393100 (52 \#13910)}

\bibitem{Ronan} Mark Ronan, \emph{Lectures on buildings}, Academic
  Press Inc., Boston, MA, 1989. \MR{90j:20001}


\bibitem{Serre-LocalFields} Jean-Pierre Serre, \emph{Local fields},
  Graduate Texts in Mathematics, vol.~67, Springer-Verlag, New York,
  1979, Translated from the French by Marvin Jay Greenberg. \MR{554237
    (82e:12016)}

\bibitem{Shemanske-Split} Thomas~R. Shemanske, \emph{Split orders and
    convex polytopes in buildings}, J.  Number Theory \textbf{130}
  (2010), no.~1, 101--115. \MR{MR2569844}

\bibitem{Vigneras-isospectral} Marie-France Vign{\'e}ras,
  \emph{Vari\'et\'es riemanniennes isospectrales et non
    isom\'etriques}, Ann. of Math. (2) \textbf{112} (1980), no.~1,
  21--32.  \MR{584073 (82b:58102)}

\end{thebibliography}

\providecommand{\bysame}{\leavevmode\hbox to3em{\hrulefill}\thinspace}
\providecommand{\MR}{\relax\ifhmode\unskip\space\fi MR }
\providecommand{\MRhref}[2]{%
  \href{http://www.ams.org/mathscinet-getitem?mr=#1}{#2} }
\providecommand{\href}[2]{#2}

\end{document}